\documentclass[pdflatex,sn-mathphys-num]{sn-jnl}

\usepackage{graphicx}%
\usepackage{multirow}%
\usepackage{amsmath,amssymb,amsfonts}%
\usepackage{amsthm}%
\usepackage{mathrsfs}%
\usepackage[title]{appendix}%
\usepackage{xcolor}%
\usepackage{textcomp}%
\usepackage{manyfoot}%
\usepackage{booktabs}%
\usepackage{algorithm}%
\usepackage{algorithmicx}%
\usepackage{algpseudocode}%
\usepackage{listings}%
\usepackage{bm}
\usepackage{subfigure}

\theoremstyle{thmstyleone}%
\newtheorem{theorem}{Theorem}
\newtheorem{lemma}[theorem]{Lemma}%

\theoremstyle{thmstyletwo}%
\newtheorem{example}{Example}%
\newtheorem{remark}{Remark}%

\theoremstyle{thmstylethree}%

\begin{document}

\title[Monotonicity-based regularization approach]{A Monotonicity-Based Regularization Approach to Shape Reconstruction for the Helmholtz Equation}


\author[1]{\fnm{Sarah} \sur{Eberle-Blick}}\email{Sarah.Eberle-Blick@ku.de}

\author[2]{\fnm{Bastian} \sur{Harrach}}\email{harrach@math.uni-frankfurt.de}

\author*[2,3]{\fnm{Xianchao} \sur{Wang}}\email{xcwang90@gmail.com}

\affil[1]{\orgdiv{Mathematical Institute for Machine Learning and Data Science}, \orgname{Catholic University of Eichst\"att-Ingolstadt},  \orgaddress{\street{Hohe-Schul-Str. 5}, \city{Ingolstadt}, \postcode{85049}, \state{Bavaria}, \country{Germany}}}

\affil[2]{\orgdiv{Institute for Mathematics}, \orgname{Goethe-University Frankfurt}, \orgaddress{\street{Robert-Mayer-Str. 10}, \city{Frankfurt am Main}, \postcode{60325}, \state{Hessen}, \country{Germany}}}

\affil*[3]{\orgdiv{School of Mathematics}, \orgname{Harbin Institute of Technology}, \orgaddress{\street{West-Dazhi-Str. 92}, \city{Harbin}, \postcode{150001}, \state{Heilongjiang}, \country{P. R. China}}}


\abstract{We consider an inverse boundary value problem for determining unknown scatterers, which is governed by the Helmholtz equation in a bounded domain.  To address this, we develop a novel convex data-fitting formulation that is capable of reconstructing the shape of the unknown scatterers.
Our formulation is based on a monotonicity relation between the scattering index and boundary measurements. We use this relation to obtain a pixel-wise constraint on the unknown scattering index, and then minimize a data-fitting functional defined as the sum of all positive eigenvalues of a linearized residual operator. The main advantages of our new approach are that this is a convex data-fitting problem that does not require additional PDE solutions. The global convergence and stability of the method are rigorously established to demonstrate the theoretical soundness. In addition, several numerical experiments are conducted to verify the effectiveness of the proposed approach in shape reconstruction.}




\maketitle

\section{Introduction}\label{sec1}

This article is concerned with an inverse boundary value problem for the  Helmholtz equation in a bounded region. Our goal is  to reconstruct  unknown scatterers inside a domain by using the measured wave field on the  boundary. Typically, Neumann boundary data are prescribed as input, and an array of receivers is placed along the boundary. The corresponding Dirichlet data measured by these receivers are then used to identify the locations and shapes of the scatterers.  This type of problem has attracted substantial attentions in both science and technology fields, such as  non-destructive testing \cite{Isakov, Liu15} and medical imaging \cite{Arridge99, Zhou18}.

Next, we present the mathematical formulation of an inverse boundary value problem for the Helmholtz equation.  Let $\Omega\subset \mathbb{R}^n$, $n\geq 2$, be a bounded Lipschitz domain with unit exterior normal vector $\nu$, and let $k>0$ be the frequency. We consider the following Helmholtz equation  with  a Neumann boundary data $g\in L^2(\partial \Omega)$
\begin{equation}\label{eq:main}
  \begin{cases}
  &\Delta u+k^2 q u=0,\quad \text{in}\ \Omega, \medskip\\
  & \partial_{\nu}u=g,  \qquad \qquad  \, \text{on}\ \partial \Omega,
  \end{cases}
\end{equation}
where $q(x)\in L^{\infty}(\Omega)$ denotes the refractive index.
We define the  Neumann-to-Dirichlet (NtD) operator by
\begin{equation}\label{eq:DtN}
  \Lambda(q): g\in L^2(\partial \Omega) \to u\in L^2(\partial \Omega).
\end{equation}
For any fixed $k>0$, to ensure the Helmholtz  problem has a unique solution,  the domain of the NtD operator is defined as
  \begin{equation}\label{eq:D}
\mathcal{D}(\Lambda):=\{ q(x):  q(x)\in L^{\infty}(\Omega)\setminus L_k(\Omega)\},
\end{equation}
where $L_k(\Omega)$ represents  the set of all $q\in L^{\infty}(\Omega)$ for which $k$ is a resonance frequency.
The forward problem  associated with \eqref{eq:main} is to determine the field $u$   for a given parameter $q$  and  boundary input $g$.  The well-posedness of this forward  problem can be conveniently found in \cite{Harrach19A-PDE}.
The inverse  problem  we are concerned with  is to determine scatterers, i.e. regions where the refractive  index  $q\in \mathcal{D}(\Lambda)$ differs from $1$, from the knowledge of finitely many measurements
\begin{equation}\label{eq:Fn}
 \mathcal{F}_n: \mathcal{D}(\Lambda)\rightarrow \mathbb{S}_n, \quad   \mathcal{F}_n(q): = \left(\int_{\partial \Omega} g_i \, \Lambda(q) g_j \,\mathrm{d}s\right)_{i,j=1}^n,
\end{equation}
where $\mathbb{S}_n\subset \mathbb{R}^{n\times n}$ denotes the space of symmetric matrices and $\{ g_1, g_2,\cdots, g_n \}$ forms an orthonormal basis in $L^2(\partial \Omega)$.

Inverse boundary value problems for the Helmholtz equation are typically nonlinear and severely ill-posed. Regarding uniqueness, Sylvester and Uhlmann \cite{Sylvester87} established global uniqueness under the assumption that the refractive index is a bounded measurable function. Furthermore, Harrach, Pohjola, and Salo \cite{Harrach19A-PDE} demonstrated local uniqueness from partial boundary data. Concerning stability, H\"ahner and Hohage \cite{Hohage01} proved that the aforementioned inverse boundary value problem  exhibits logarithmic stability. Subsequently, Beretta {\it et al} \cite{Beretta16} provided a Lipschitz stability estimate under the assumption that the refractive index can be represented as  a linear combination of piecewise constant functions.
With respect to numerical methods, optimization techniques are widely used for reconstructing unknown scatterers. These approaches involve solving a minimization problem, where the objective function measures the discrepancy between measured data and numerical solutions \cite{Haddar04}. However, they typically require a good initial guess for the unknown parameters as prior information, and the quality of the reconstruction strongly depends on the choice of regularization parameters. Moreover, such iterative methods lack theoretical guarantees for convergence, and the associated optimization process is computationally expensive.
In contrast, non-iterative methods, such as the linear sampling method \cite{Cakoni11, Colton1996}, factorization method \cite{Kirsch, Kirsch1998}, probe method \cite{Potthast06},   direct sampling method \cite{Ito12, chen2016direct},  and the recently developed monotonicity method \cite{Tamburrino2002, Griesmaier2018,   Harrach13SIMA, eberle2024monotonicity, lin2025montonicity, garde2024reconstruction, arens2023monotonicity, albicker2023monotonicity, mottola2025inverse, esposito2024piecewise, harrach2025monotonicity}, offer rigorous theoretical justifications and do not require  initial guesses or regularization parameters.
Nevertheless, these non-iterative approaches rarely achieve high reconstruction resolution. To address these challenges, Harrach and Minh \cite{Harrach16IP, harrach2018monotonicity} proposed a monotonicity-based regularization method for solving the inverse conductivity problem by combining the monotonicity method with data-fitting functionals. This method not only achieves good reconstruction performance but also guarantees global convergence. Later, Eberle and Harrach \cite{Eberle2022} extended this technique to inverse boundary value problems for stationary linear elasticity. It is worth noting that monotonicity-based regularization methods have so far been applicable only to stationary models.

In this work,  we  extend the framework of the monotonicity-based regularization method to the frequency-domain Helmholtz equation in a bounded domain. It is worth noting that this extension is nontrivial and involves subtle and technical analysis.  To highlight the novel contributions of the current study, we present three remarks as follows.
 Firstly, the monotonicity-based regularization technique has been developed primarily for stationary problems \cite{Harrach16IP, Eberle2022} and cannot be directly applied to frequency-domain problems,  such as the Helmholtz equation.  The main difficulty arises from the fact that the Helmholtz equation is not coercive but a compact perturbation of a coercive equation.  To address this issue, we provide a rigorous proof demonstrating that the domain of the Neumann-to-Dirichlet operator (excluding the eigenvalues)  is an open set. Then we prove that the Neumann-to-Dirichlet operator  is  Fr\'echet  differentiable in this domain  and provide  its Fr\'echet derivative in quadratic form.
  Furthermore, we establish a monotonicity relation between the refractive index and the Neumann-to-Dirichlet operator via  its Fr\'echet derivative.
Secondly, in contrast to the conductivity equation, the monotonicity relation for the Helmholtz equation holds only up to subspaces with finite codimension, specifically less than the number of positive Neumann eigenvalues \cite{Harrach19A-PDE}.  While the upper bound for the codimension of the subspaces has been improved in \cite{Harrach19siap}, determining an optimal  upper bound  remains an open issue. To tackle this challenge, we consider a circular model with a two-layer symmetric and piecewise constant medium, and numerically estimate the upper bound for the codimension of the subspaces.
Finally,  the monotonicity-based regularization method in \cite{Harrach16IP}, is formulated by minimizing a linearized residual functional with respect to the Frobenius norm. However, since the monotonicity relation for the Helmholtz equation is established only in subspaces of finite codimension, minimizing the linearized residual functional with respect to the Frobenius norm may fail in this case. To address this issue, we propose a novel strategy that minimizes the sum of the positive eigenvalues  of the linearized residual functional, which is solved by using the semidefinite programming.  Noting that there is no constraint on the corresponding negative eigenvalues, the minimization problem may fail to converge in the presence of measurement noise.  To enhance the stability of the approach, we also add the Frobenius norm of the linearized residual functional as a penalty term, with the noise level serving as the regularization parameter.  Therefore, the theoretical justifications for convergence and stability are significantly more challenging  than those for stationary models.

The promising features of the proposed monotonicity-based regularization approach for the Helmholtz equation can be summarized in two key aspects. On the one hand, the method offers high computational efficiency and low computational cost. Specifically, at each iteration step, there is no need to solve the forward problem, as the algorithm employs a single-step linearization and retains a fixed Fr\'echet derivative throughout the entire iterative process. Furthermore, although the method is based on minimizing the sum of all positive eigenvalues of the residual functional, this quantity is efficiently evaluated via semidefinite programming (SDP), thereby avoiding explicit eigenvalue computations.
On the other hand, in contrast to other optimization methods that offer only local convergence, the proposed method guarantees global convergence and exhibits robust performance. To the best of our knowledge, this is the first attempt to minimize the sum of all positive eigenvalues with a guarantee of global convergence, owing to the convexity of the objective functional.   Interested readers are also referred to the globally convergent algorithms and convexification techniques presented in \cite{Harrach21, Klibanov19, klibanov2025convexification, le2023numerical}.
Notably, the method does not require prior knowledge of the refractive index of the unknown scatterer or an accurate initial guess, only the background refractive index is assumed to be known.  Furthermore, the method demonstrates strong stability, maintaining reliable performance even under 10\% and 1\% measurement noise for single and multiple scatterers, respectively.

The rest of the paper is organized as follows. In Section 2, we first show that the set  $L_k(\Omega)$ defined in \eqref{eq:D}, is closed by applying the Banach-Alaoglu theorem. As a consequence, the domain of the Neumann-to-Dirichlet operator
$\Lambda(q)$, excluding the eigenmodes, forms an open set. Next, we prove that the Neumann-to-Dirichlet operator is Fr\'echet differentiable and derive its Fr\'echet derivative in quadratic form. We then establish a monotonicity relation between the refractive index
$q$ and the Neumann-to-Dirichlet operator $\Lambda(q)$. Finally, we propose a strategy to estimate the upper bound of the number of negative eigenvalues based on numerical calculations for piecewise constant layers. In Section 3, we focus on the monotonicity-based method for reconstructing the shape of unknown scatterers. First, we introduce the residual functional obtained via single-step linearization. Second, we present the associated minimization problem, where the objective functional is defined by the sum of  all positive eigenvalues and the constraint is formulated using a monotonicity relation, serving as a form of regularization. Third, we provide theoretical justifications for global convergence and stability, demonstrating the effectiveness of the proposed method. Several numerical experiments are presented in Section 4 to validate the performance of the proposed approach.
 In particular, we provide a selection rule for determining the upper bound of the regularization term.

\section{A monotonicity relation for the Helmholtz equation}

In this section, we establish a monotonicity relation between the refractive index $q$ and the Neumann-to-Dirichlet (NtD) operator $\Lambda(q)$, which is also related to  the Fr\'echet derivative of the NtD operator.  It is noted that  the openness of the domain $\mathcal{D}(\Lambda)$  is  a necessary condition for the NtD operator $\Lambda(q)$  to be  Fr\'echet  differentiable.   To this end, we first demonstrate that $\mathcal{D}(\Lambda)$  is an open set.


\begin{lemma}\label{lem:open-set}
Let  $\Lambda(q)$  be the Neumann-to-Dirichlet operator defined in \eqref{eq:DtN}, associated with  the Helmholtz equation \eqref{eq:main}. Then  the domain $\mathcal{D}(\Lambda)$  defined in \eqref{eq:D}  is an open set.
\end{lemma}
\begin{proof}
According to  formula  \eqref{eq:D} and the definition of open sets,    it suffices to show that  $L_k(\Omega)$ is a closed set.
To this end, assume there exists a sequence $\{q_n\}\subset L_k(\Omega)$ such that $q_n\rightarrow q^*$ in $L^{\infty}(\Omega)$. We then prove that   $q^* \in L_k(\Omega)$.

Given $q_n(x)\in L_k(\Omega)$, there exists a non-trivial solution $u_n\in H^1(\Omega)$ to the homogeneous equation
\begin{equation}\label{eq:homo}
\begin{cases}
  &\Delta u_n+k^2 q_n u_n=0,\quad \text{in}\ \Omega, \medskip\\
  & \partial_{\nu}u_n=0,  \qquad \qquad \quad  \, \text{on}\ \partial \Omega.
  \end{cases}
\end{equation}
Without loss of generality, we assume $\|q_n\|_{L^{\infty}(\Omega)}\leq 2 \|q^*\|_{L^{\infty}(\Omega)}$, and that  $u_n$ is normalized in $L^2(\Omega)$, such that $\|u_n\|_{L^2(\Omega)}=1$.
Since the last homogeneous equation \eqref{eq:homo}  satisfies the following  variational form
\begin{equation*}
    \int_{\Omega} |\nabla u_{n}|^2-k^2 q_{n} |u_{n}|^2\, \mathrm{d}x=0,
\end{equation*}
one can derive  that
\begin{equation*}
  \|u_n\|_{H^1(\Omega)}=\|u_n\|_{L^2(\Omega)}+\|\nabla u_{n}\|_{L^2(\Omega)}\leq 1+2 k^2 |\Omega|\cdot \|q^*\|_{L^{\infty}(\Omega)} < \infty,
\end{equation*}
 where $|\Omega|$ denotes the volume of $\Omega$. Thus, the sequence $\{u_n\}$ is uniformly bounded in $H^1(\Omega)$.
 By the Banach-Alaoglu theorem \cite{Brezis}, there exists a subsequence $u_{n_s}\in H^1(\Omega)$ and a function $u^*\in H^1(\Omega)$ such that $u_{n_s}\rightharpoonup u^*$   weakly in $H^1(\Omega)$. By the weak convergence $u_{n_s}\rightharpoonup u^*$ in $H^1(\Omega)$,  and the strong convergence $q_{n_s}\rightarrow q^*$ in  $L^{\infty}(\Omega)$ ($\{q_{n_s}\}$ is the corresponding subsequence of $\{q_n\}$),  we obtain
\begin{equation*}
    \int_{\Omega} (\nabla u_{n_s}\cdot \nabla v-k^2 q_{n_s}u_{n_s}v )\, \mathrm{d}x=0, \quad \text{for\,all}\ v\in H^1(\Omega).
\end{equation*}
Passing the limit as $n_s \rightarrow \infty$, one has
\begin{equation*}
    \int_{\Omega} (\nabla u^*\cdot \nabla v-k^2 q^* u^* v )\, \mathrm{d}x=0, \quad \text{for\,all}\ v\in H^1(\Omega),
\end{equation*}
which implies that $ u^*$ is a solution to  the following homogeneous equation
\begin{equation*}
\begin{cases}
  &\Delta u +k^2 q^* u=0,\quad \text{in}\ \Omega, \medskip\\
  & \partial_{\nu}u=0,  \qquad \qquad  \, \text{on}\ \partial \Omega.
  \end{cases}
\end{equation*}
  Furthermore, due to the compact embedding of $H^1(\Omega)$ into $L^2(\Omega)$, one can find that  $u_{n_s}\rightarrow u^*$   strongly in $L^2(\Omega)$. It follows that $u^*$ is not identically zero, as $\|u_{n_s}\|_{L^2(\Omega)}=1$. Therefore,  $u^*$ is a non-trivial solution to the above homogeneous problem corresponding to $q^*$,  which implies that $q^*\in L_k(\Omega)$.
\end{proof}

Next,  we characterize the Fr\'echet derivative of the NtD operator $\Lambda(q)$ associated with the Helmholtz equation in the following lemma.

\begin{lemma}\label{eq:Frechet}
Let $q\in \mathcal{D}(\Lambda)$, then the  NtD operator $\Lambda(q)$ is   Fr\'echet  differentiable and its  derivative
$\Lambda' (q)h\in \mathcal{L}(L^2(\partial \Omega))$ is  associated with  the quadratic form
\begin{equation*}
\int_{\partial \Omega}  g (\Lambda' (q)h) g \, \mathrm{d}s=\int_{\Omega} k^2 h |u|^2 \,\mathrm{d}x.
\end{equation*}
where $h\in L^{\infty}(\Omega)$ denotes a shift term with respect to $q$.
\end{lemma}
\begin{proof}
The equivalent weak formulation of \eqref{eq:main} is to find $u\in H^1(\Omega)$
such that
\begin{equation}\label{eq:variation}
  \int_{\Omega} (\nabla u\cdot \nabla v-k^2 quv )\, \mathrm{d}x=\int_{\partial \Omega} g v \,\mathrm{d}s,\quad \text{for\,all}\ v\in H^1(\Omega).
\end{equation}
Let $ a(q; u,v)$ denote the left-hand side of the variational formula \eqref{eq:variation}.  It can be expressed   as a sum of a trilinear  form $b(q;u,v)$ and a bilinear form $c(u,v)$,
\begin{equation*}\label{eq:variational}
a(q; u,v):=b(q;u,v)+c(u,v),
\end{equation*}
where
\begin{equation}\label{eq:linear-form}
b(q;u,v)= - \int_{\Omega} k^2 quv\, \mathrm{d}x,  \quad  \text{and} \quad c(u,v)= \int_{\Omega} \nabla u\cdot \nabla v \, \mathrm{d}x.
\end{equation}
 Since $\mathcal{D}(\Lambda)$ is an open set by Lemma \ref{lem:open-set},  it follows  from \cite[Lemma 2.1]{Lechleiter08} that the NtD operator  $\Lambda(q)$  is  Fr\'echet  differentiable,  and its Fr\'echet derivative
$\Lambda' (q)$ satisfies
\begin{equation*}
  a(q; (\Lambda' (q)h)g,v)=-b(h; \Lambda(q)g,v), \quad \text{for\,all}\ v\in H^1(\Omega).
\end{equation*}

Noting that $u=\Lambda(q) g$,  and using \eqref{eq:variation} with $v=u$, one can find that
\begin{equation*}
a(q;\Lambda(q) g, u)=\int_{\partial \Omega} g u \, \mathrm{d}s= \int_{\partial \Omega} g \Lambda(q) g \, \mathrm{d}s.
\end{equation*}
Furthermore,  replacing  $\Lambda(q)$ with $\Lambda'(q)h$, we obtain
\begin{equation}\label{eq:a2}
a(q; (\Lambda'(q)h) g, u)= \int_{\partial \Omega} g (\Lambda'(q)h) g \, \mathrm{d}s.
\end{equation}
Thus, by combining equations \eqref{eq:linear-form}--\eqref{eq:a2},  we  deduce that
\begin{equation*}
\int_{\partial \Omega}  g (\Lambda' (q)h) g \, \mathrm{d}s=a(q; (\Lambda'(q)h) g, u)=-b(h; \Lambda(q)g, u)=\int_{\Omega} k^2 h |u|^2 \,\mathrm{d}x,
\end{equation*}
which completes the proof.
\end{proof}

With the Fr\'echet derivative of NtD operator presented in Lemma \ref{eq:Frechet},  we now establish the monotonicity relationship between the refractive index $q$ and the NtD operator $\Lambda(q)$.

\begin{theorem}\label{monotonicity-relation}
Let  $q \in  \mathcal{D}(\Lambda)$ be the refractive index of the scatterer   and let  $d(q)\in \mathbb{N}$  be the number of negative Neumann eigenvalues of the operator
$-(\Delta+k^2 q)$. Then there exists a subspace $V_1\subset L^2(\partial \Omega)$ with $\dim(V_1)\leq  d(q)$ such that
\begin{equation*}\label{eq:monoto1}
   \int_{\partial \Omega} g \left(\Lambda (q)-\Lambda (q_0)-\Lambda' (q_0)(q-q_0) \right)g \,  \mathrm{d}s \geq 0, \quad \text{for all}\ g\in V_1^{\bot},
\end{equation*}
where $q_0$ denotes the background refractive index.
\end{theorem}

\begin{proof}

Let $u_{q_0}$ be the solution to the following equation
 \begin{equation*}
  \begin{cases}
  &\Delta u+k^2 q_0 u=0,\quad \text{in}\ \Omega, \medskip\\
  & \partial_{\nu}u=g,  \qquad \qquad  \,  \, \, \text{on}\ \partial \Omega.
  \end{cases}
\end{equation*}
From \cite[Theorem 3.5]{Harrach19A-PDE}, one can find that  there exists a subspace $V_1\subset L^2(\partial \Omega)$ with $\dim(V_1)\leq  d(q)$ such that
\begin{equation*}
   \int_{\partial \Omega} g \left(\Lambda (q)-\Lambda (q_0) \right) g \,  \mathrm{d}s \geq   \int_{ \Omega}  k^2 (q-q_0) \left|u_{q_0}\right|^2 \,  \mathrm{d}s, \quad \text{for all}\ g\in V_1^{\bot}.
\end{equation*}
Therefore, combining the last formula and  Lemma \ref{eq:Frechet},  one can find that
 \begin{equation*}
   \int_{\partial \Omega} g \left(\Lambda (q)-\Lambda (q_0)-\Lambda' (q_0)(q-q_0) \right)g \,  \mathrm{d}s \geq 0, \quad \text{for all}\ g\in V_1^{\bot},
\end{equation*}
 which completes the proof.
\end{proof}

\begin{remark}
According to Theorem \ref{monotonicity-relation}, by interchanging the refractive index $q_0$ and $q$,  one can deduce that there exists a subspace  $V_2\subset L^2(\partial \Omega)$ with $\dim(V_2)\leq  d(q_0)$  such that
\begin{equation}\label{eq:monoto2}
   \int_{\partial \Omega} g \left(\Lambda (q)-\Lambda (q_0)-\Lambda' (q)(q-q_0) \right)g \,  \mathrm{d}s \leq 0, \quad \text{for all}\ g\in V_2^{\bot}.
\end{equation}

\end{remark}

\begin{remark}
Theorem \ref{monotonicity-relation} states that  the operator $\Lambda (q)-\Lambda (q_0)-\Lambda' (q_0)(q-q_0)$  has finitely many negative eigenvalues,  with an upper bound of $d(q)$.
But   the bound on $d(q)$ is typically unavailable in practice since the refractive index $q$ is  unknown.  In fact, as shown in \cite[Corollary 3.11]{Harrach19A-PDE},  one can take $q_{\max}$ to be the essential supremum of  $q(x)$, so that  $d(q_{\max})$  provides an upper bound for
$d(q)$.  However, this bound  is excessively large and may reduce the accuracy of the inversion. To address this  issue,  we will introduce a strategy  to estimate the  upper bound for $d(q)$ in Section \ref{subsec:3.1}.
\end{remark}

\section{A Monotonicity-based regularization method}

In this section,   we introduce a  monotonicity-based regularization method to determine the shape of the unknown scatterer  from the measured data $\mathcal{F}_n(q)$, as defined in \eqref{eq:Fn},  by  utilizing  the monotonicity relation established in Theorem  \ref{monotonicity-relation}.
Rigorous theoretical justifications of the convergence and stability are provided to verify the validity of the proposed method. Before our discussion, we introduce a model setting used to estimate an upper bound for $d(q)$.

\subsection{Model setting}\label{subsec:3.1}

  For the Helmholtz equation \eqref{eq:main},
we assume that the background refractive index $q_0\in \mathbb{R}_+$ is known as {\it a priori} information and the refractive index of the investigated object is defined as $q:=q_0+\gamma \chi_D$, where $\overline{D}\subset \Omega$ represents the shape of the  target and $\chi_{D}$ represents the characteristic function of $D$.  Here,  we further assume that  $\gamma\in L_+^{\infty}(D)$, where $L_+^{\infty}(D)$ is a subspace of $L^{\infty}(D)$ with positive essential infima.
Without loss generality, we suppose that
\begin{equation*}
\begin{aligned}
  &q_0< q_{\min}\leq q(x) \leq q_{\max},  \quad \text{in}\  D,\\
  & q(x)=q_0, \qquad \qquad\qquad \quad \  \text{in}\ \Omega \backslash D.
  \end{aligned}
\end{equation*}
Here $q_{\min}$ and $q_{\max}$ respectively represent as the infimum and supremum of $q(x)$.
Since the exact shape $D$ is also unknown, we assume the existence of a domain $\Omega_0\subset \Omega$ such that
 $D\subset \Omega_0$. Furthermore, we consider the following Helmholtz equation
\begin{equation}\label{eq:modified-Helmholtz}
  \Delta u +k^2 \widetilde{q} u=0,  \quad \text{in}\ \Omega,
\end{equation}
where the refractive index is assumed to be a radially symmetric, piecewise constant function
\begin{equation}\label{eq:q-tilde}
 \widetilde{ q}(x)=
  \begin{cases}
  q_{\max},   &x\in \Omega_0\subset \Omega, \medskip\\
  q_0, & x\in \Omega\backslash\Omega_0.
  \end{cases}
\end{equation}

According to  \cite[Lemma 3.9]{Harrach19A-PDE}, one can find that $d(q)\leq d(\widetilde{q})\leq d(q_{\max})$ whenever  $q(x)\leq \widetilde{q}(x)\leq q_{\max}$ for all $x\in \Omega$. Therefore,  a sharper upper bound $d(\widetilde{q})$ for $d(q)$ can be obtained.
Numerically,  one can utilize the finite element method  to  determine $d(\widetilde{q})$ by  computing the number of negative eigenvalues of the  stiffness matrix associated with the operator $-(\Delta+k^2 \widetilde{q})$,  since $d(\widetilde{q})$  is the number of negative Neumann eigenvalues of this operator.

\subsection{Monotonicity-based regularization}

In order to reconstruct the difference $q-q_0$ or $\gamma \chi_D$,  we compare the measurements $\mathcal{F}_n(q)$  with   $\mathcal{F}_n(q_0)$ for the  background refractive index $q_0$.
To this end, we apply a single linearization step
\begin{equation}\label{eq:appro-F}
\mathcal{F}'_n (q_0) (q-q_0)\approx \mathcal{F}_n(q) - \mathcal{F}_n(q_0),
\end{equation}
where $\mathcal{F}'_n (q_0): L^{\infty}(\Omega)\to \mathbb{S}_n\subset \mathbb{R}^{n\times n}$ denotes the Fr\'echet derivative of $ \mathcal{F}_n (q)$ at $q_0$, given by
\begin{equation}\label{eq:Fn-derivative}
 \mathcal{F}'_n(q_0): h\mapsto  \left(\int_{\Omega} k^2 h\,  u_{q_0}^{(g_i)}   u_{q_0}^{(g_j)} \,\mathrm{d}x \right)_{i,j=1}^n.
\end{equation}
To solve the above problem, we discretize the whole domain  $\overline{\Omega}=\bigcup_{m=1}^M \overline{P}_m$ into $M$ disjoint open pixels $P_m$. We approximate the difference $h\approx q-q_0$ using a piecewise constant ansatz
\begin{equation*}
   h=\sum_{m=1}^M a_m \chi_{P_m} (x).
\end{equation*}
Then the formula \eqref{eq:appro-F} can be rewritten as the following linear equation
\begin{equation*}
\mathcal{F}'_n (q_0) h= \mathcal{F}_n(q) - \mathcal{F}_n(q_0).
\end{equation*}
Furthermore, we define the residual functional as
\begin{equation*}
  \bm R(h):=\mathcal{F}_{n}(q)-\mathcal{F}_{n}(q_0)- \mathcal{F}_{n}'(q_0) h.
\end{equation*}

Next, we present the  monotonicity-based regularization approach for reconstructing the shape of the unknown target.   The main idea of the monotonicity-based regularization approach is to minimize the linearized residual functional subject to a constraint on $h$, which is derived from the monotonicity properties established in Theorem~\ref{monotonicity-relation}.
Let $\lambda_1 (  \bm R(h) )\geq \lambda_2(  \bm R(h) )\geq \ldots \geq \lambda_n(  \bm R(h) )$ denote the $n$ eigenvalues of $\bm R (h)$ in descending order. To  determine $h$,  we  consider the following minimization problem with the sum of all  positive eigenvalues of $\bm R (h)$ as:
\begin{equation}\label{eq:minimizer}
  \min_{h\in \mathcal{A}} \left(  \sum_{\{j: \, \lambda_j>0\}} \lambda_j(R(h)) \right),
 \end{equation}
where $\mathcal{A}$ denotes an admissible set of $h$, defined by
\begin{equation*}
  \mathcal{A}:=\left\{h\in L^{\infty}(\Omega): \ h=\sum_{m=1}^M a_m \chi_{P_m} (x),  \ 0\leq a_m\leq \min(q_{\min}-q_0, \beta_m)   \right\},
\end{equation*}
with
 \begin{equation}\label{eq:beta}
   \beta_m:=\max\left \{\alpha\geq 0:    \Lambda(q)-  \Lambda(q_0)- \alpha \Lambda'(q_0) \chi_{P_m}  \geq_{d(\widetilde{q})} 0\right\}.
 \end{equation}
Here  the  linear constraint $\beta_m$ is defined by a monotonicity test, which also serves as a special regularizer.

\begin{remark}
In the minimization problem \eqref{eq:minimizer},  we minimize the residual functional using the sum of all positive eigenvalues, rather than the Frobenius norm of the operator $\bm R(h)$.  This choice is motivated by the fact that minimizing the Frobenius norm does not guarantee global convergence for the Helmholtz equation, in contrast to the cases of electrical impedance tomography (EIT) and stationary elasticity  \cite{Harrach16IP,  Harrach21, Eberle2022}.
\end{remark}

\subsection{Convergence analysis}
 In what follows, we  show that the support of the minimizer  $h$ in \eqref{eq:minimizer} agrees with the scatterer's shape $D$.  Before our discussion, we introduce the following two lemmas.

\begin{lemma}\label{lemma1}
Let $\beta_m$ be defined in \eqref{eq:beta}. For any pixel $P_m$, it holds that $P_m\subseteq D$ if and only if $\beta_m>0$.
Moreover, for $P_m\subseteq D$, it holds that $\beta_m\geq q_{\min}-q_0$.
\end{lemma}
\begin{proof}
Step 1:
We prove that $P_m\subseteq D$ implies  $\beta_m\geq q_{\min}-q_0>0$.
Let $P_m\subseteq D$, all $\alpha\in [0,\,q_{\min}-q_0 ]$. By Theorem~\ref{monotonicity-relation} there exists a subspace $V\subset L^2(\partial \Omega)$ with $\dim(V)\leq  d(q)\leq d(\widetilde{q})$,
so that, for all $g\in V^{\bot}$,
\begin{equation*}
\begin{aligned}
 &\int_{\partial \Omega} g (\Lambda (q)-\Lambda (q_0)- \Lambda'(q_0) \alpha \chi_{P_m} )g\,  \mathrm{d}s\\
&\quad\geq \int_{\partial \Omega} g ( \Lambda'(q_0)(q-q_0) - \Lambda'(q_0) \alpha \chi_{P_m} )g\,  \mathrm{d}s\\
&\quad\geq \int_{D} k^2 (q-q_0) |u_{q_0}^{(g)}|^2\, \mathrm{d}x - \int_{P_m} k^2 (q_{\min}-q_0) |u_{q_0}^{(g)}|^2\, \mathrm{d}x\\
&\quad \geq  0.
  \end{aligned}
\end{equation*}
Hence,  according to  the definition of $\beta_m$,  one can find that $\beta_m \geq q_{\min}-q_0>0$.

Step 2:  We will demonstrate that if  $\beta_m>0$, then $P_m\subseteq D$;  conversely, if $P_m \not \subseteq D$, then $\beta_m=0$.
By the definition of $\beta_m$,  this is equivalent  to the statement that   for all $\alpha>0$,  the inequality $\Lambda(q)-\Lambda(q_0)-\alpha \Lambda'(q_0)\chi_{P_m} \not \geq_{d(\widetilde q)}  0$ holds.  In what follows, we prove this by contradiction: assume that $P_m \not \subseteq D$,  there exists $\alpha>0$ such that $\Lambda(q)-\Lambda(q_0)-\alpha \Lambda'(q_0)\chi_{P_m}  \geq_{d(\widetilde q)}  0$.
Then there exists a subspace $V\subset L^2(\partial \Omega)$ with $\dim(V)=d(\widetilde q)$ such that for all $g\in V^{\perp}$, it holds that
\begin{equation}\label{eq:monoto-c1}
 \int_{\partial \Omega} g (\Lambda (q)-\Lambda (q_0)-\Lambda'(q_0) \alpha  \chi_{P_m} )g\,  \mathrm{d}s \geq 0.
\end{equation}
By combining  inequalities \eqref{eq:monoto-c1} and  \eqref{eq:monoto2}, there exists a subspace $\widetilde{V}\subset L^2(\partial \Omega)$  with $ V \subset \widetilde{V} $  and  $d(\tilde q)\leq \dim(\widetilde V)\leq d(\tilde q)+ d(q_0)$, such that
\begin{equation}\label{eq:con0}
\begin{aligned}
 0&\leq \int_{\partial \Omega}  g (\Lambda (q)-\Lambda (q_0)-\Lambda'(q_0) \alpha  \chi_{P_m} ) g\,  \mathrm{d}s \\
 &\leq   \int_{\Omega} k^2 (q-q_0) |u_q^{(g)}|^2\, \mathrm{d}x -\int_{\Omega} k^2 \alpha \chi_{P_m} |u_{q_0}^{( g)}|^2\, \mathrm{d}x\\
 &\leq  k^2 (q_{\max}-q_0)  \int_{D} |u_q^{( g)}|^2\, \mathrm{d}x- k^2 \alpha \int_{P_m} |u_{q_0}^{(g)}|^2\, \mathrm{d}x
\end{aligned}
\end{equation}
 for all $g\in \widetilde V^{\perp}$.  By \cite[Theorem 4.2]{Harrach19A-PDE},  there exists a constant $C>0$ such that
\begin{equation}\label{eq:con1}
  k^2 (q_{\max}-q_0)  \int_{D} |u_{q}^{({g})}|^2\, \mathrm{d}x\leq C  k^2 (q_{\max}-q_0)  \int_{D} |u_{q_0}^{({g})}|^2\, \mathrm{d}x
\end{equation}
for all  $g\in \widetilde V^{\perp}$. Combining  inequalities \eqref{eq:con0} and \eqref{eq:con1}, one can find that
\begin{equation}\label{eq:con2}
   0\leq  C k^2 (q_{\max}-q_0)  \int_{D} |u_q^{( g)}|^2\, \mathrm{d}x- k^2 \alpha \int_{P_m} |u_{q_0}^{(g)}|^2\, \mathrm{d}x
\end{equation}
for all  $g\in \widetilde V^{\perp}$. Noting that $P_m \not \subset D$, according to localized potentials \cite[Theorem 4.1]{Harrach19A-PDE},  there exists a sequence  $g_m\in  V^{\bot}$ such that
\begin{equation*}
  \lim_{m\rightarrow \infty}  \int_{P_m} |u_{q_0}^{(g_m)}|^2\, \mathrm{d}x \rightarrow \infty, \quad   \lim_{m\rightarrow \infty}  \int_{D} |u_{q_0}^{(g_m)}|^2\, \mathrm{d}x \rightarrow 0.
\end{equation*}
This contradicts the inequality \eqref{eq:con2} and the proof completes.
\end{proof}

\begin{lemma}\label{lemma2}
 For all pixels $P_m$,  the matrix $\mathcal{F}'_n (q_0)\chi_{P_m}\subset \mathbb{R}^{n\times n}$  is a positive definite matrix.
\end{lemma}
\begin{proof}
For all $v=(v_1, v_2, ... , v_n)^{\top}\in \mathbb{R}^n$,  using formula \eqref{eq:Fn-derivative} and Theorem \ref{eq:Frechet}, one can deduce that
\begin{equation*}
  v^{\top} (\mathcal{F}'_n (q_0)\chi_{P_m}) v=\sum_{i, j=1}^n  v_i v_j \langle g_i,  \Lambda'(q_0) \chi_{P_m}g_j \rangle =\langle g,  \Lambda'(q_0) \chi_{P_m}g \rangle
= \int_{P_m} k^2 |u_{q_0}^{(g)}|^2\, \mathrm{d}x\geq 0,
\end{equation*}
where $g=\sum_{i=1}^n v_i g_i$.
The above formula shows that $ \mathcal{F}'_n (q_0)\chi_{P_m} \subset \mathbb{R}^{n\times n}$ is a positive semi-definite symmetric matrix.

To prove definiteness, we argue by contradiction, and assume that there exists $0\neq \tilde v\in \mathbb{R}^n$ with $(\mathcal{F}'_n (q_0)\chi_{P_m}) \tilde v=0$.
Then, by linear independence, and as above,
\begin{equation*}
\tilde g=\sum_{i=1}^n \tilde v_i g_i\neq 0,\quad \text{ and } \quad \int_{P_m} k^2 |u_{q_0}^{(\tilde g)}|^2\, \mathrm{d}x= 0.
\end{equation*}
From the last equation, one can find that  $u_{q_0}^{(\tilde{g})}\equiv 0$ in $P_m$. Note that  $P_m \subset \Omega$  is an open set,  the analyticity property \cite[Theorem 2.2]{Colton2019} of  the Helmholtz equation also  implies that $u_{q_0}^{(\tilde{g})}\equiv 0$  in $\Omega$ and hence $\tilde g=0$ on $\partial \Omega$. This contradiction implies that $\mathcal{F}'_n (q_0)\chi_{P_m}$  is a positive definite matrix.
\end{proof}

We now present the convergence result for the minimization problem \eqref{eq:minimizer}.
\begin{theorem}\label{thm:main}
Consider the minimization problem
\begin{equation}\label{problem}
  \min_{h\in \mathcal{A}} \left( \sum_{\{j:\, \lambda_j>0\}} \lambda_j(\bm R(h)) \right),
 \end{equation}
where $\bm R (h)$ is a $n\times n$ symmetric  matrix depending on the parameter $h\in \mathcal{A}$.
Then  it holds that:
\begin{description}
  \item[(i)] Problem \eqref{problem} admits a unique minimizer $\hat h$.
  \item[(ii)] For any pixel $P_m$,  it holds that $P_m\subseteq \text{supp}\, \hat h $ if and only if $P_m\subseteq D$.\\
 \quad \quad ~~~~ Moreover, $\hat h=\sum_{m=1}^M \min(q_{\min}-q_0, \beta_m) \chi_{P_m}$.
\end{description}

\end{theorem}
\begin{proof}
(i) Without loss of generality,  we assume that $\lambda_1 (  \bm R(h) )\geq \lambda_2(  \bm R(h) )\geq \ldots \geq \lambda_n(  \bm R(h) )$ denote the $n$ eigenvalues of the symmetric $n\times n$ matrix $\bm R (h)$ in descending order.
We first show that $h\in \mathcal A$ implies $\lambda_j(\bm R(h))\geq 0$ for all $j=1,\ldots,n-d(\widetilde{q})$. By Theorem~\ref{monotonicity-relation} we have that
\begin{align*}
\Lambda(q)-\Lambda(q_0)-\Lambda'(q_0)h\geq_{d(\tilde q)} \Lambda'(q_0)(q-q_0-h)\geq 0.
\end{align*}
This shows that $\Lambda(q)-\Lambda(q_0)-\Lambda'(q_0)h$ can have at most $d(q)\leq d(\tilde{q})$ negative eigenvalues.  Thus, this also holds for its Galerkin projection
$\bm R(h)=\mathcal{F}_{n}(q)-\mathcal{F}_{n}(q_0)- \mathcal{F}_{n}'(q_0) h$.

Due to  $\lambda_j(\bm R(h))\geq 0$ for all $j=1,\ldots,n-d(\widetilde{q})$,  the sum of all positive eigenvalues of $\bm R(h)$  can be equivalently expressed as
\begin{equation}\label{eq:Max}
 \Phi(h):= \sum_{\{j:\, \lambda_j>0\}} \lambda_j(\bm R(h)) = \max_{s=0,\ldots,d(\tilde q)}  \sum_{j=1}^{n-s} \lambda_j(\bm R(h)).
\end{equation}

Since the matrix-valued function  $\bm R(h)\in \mathbb{S}_n$  depends continuously on $h$, each eigenvalue $\lambda_j(\bm R(h))$ is also a continuous function of $h$. As $\bm R(h)$ is a symmetric matrix with eigenvalues ordered in descending order,  the sum of  its $(n-s)$ largest eigenvalues (partial sum)  also varies continuously with $h$.  Note that the function $ \Phi(h)$ takes the maximum over a finite set of such partial sums, the mapping $h \mapsto \Phi(h)$ is continuous. Therefore, $\Phi(h)$ attains a minimum on the compact set $\mathcal{A}$.  The uniqueness of $\hat h$ will follow from the proof of (ii).

(ii) We first show that if $\hat h=\sum_{m=1}^M  \hat a_m \chi_{P_m} (x)\in \mathcal A$
is a minimizer of problem \eqref{problem},
then $\hat h=\sum_{m=1}^M  \min(q_{\min}-q_0, \beta_m) \chi_{P_m} (x)$.  According the definition of $\hat h $, it is obvious that  $\hat h\leq \sum_{m=1}^M  \min(q_{\min}-q_0, \beta_m) \chi_{P_m} (x)$.  In what follows, we show that there is a contradiction for  $\hat h< \sum_{m=1}^M  \min(q_{\min}-q_0, \beta_m) \chi_{P_m} (x)$.

If there exists a pixel $P_m$ such that  $\hat h(x)<\min(q_{\min}-q_0, \beta_m)$, we choose $\tau>0$ such that $\hat h+\tau \chi_{P_m}= \min(q_{\min}-q_0, \beta_m)$. Furthermore, we will show that
\begin{equation}\label{eq:contrad}
  \Phi(\hat h+\tau \chi_{P_m} ) <  \Phi(\hat h),
\end{equation}
which contradicts the minimality of $\hat h$.
Let $\lambda_1(\hat h) \geq  \lambda_2(\hat h) \geq \cdots \geq \lambda_n(\hat h)$ be $n$ eigenvalues of $\bm R (\hat h)$ and
$\lambda_1(\hat h+\tau \chi_{P_m} ) \geq  \lambda_2(\hat h+\tau \chi_{P_m}) \geq   \cdots  \geq \lambda_n(\hat h+\tau \chi_{P_m})$ be $n$ eigenvalues of $\bm R(\hat h+\tau \chi_{P_m})$.
Since
\begin{equation*}
\bm{R}(\hat h+ \tau \chi_{P_m})=\bm{R}(\hat h) -\tau \mathcal{F}'_n (q_0)\chi_{P_m},
\end{equation*}
together with the fact that $\mathcal{F}'_n (q_0)\chi_{P_m}$ is a positive definite matrix in Lemma \ref{lemma2},  we can apply Weyl's inequalities to get that
\begin{equation*}
  \lambda_i(\hat h+\tau \chi_{P_m})<  \lambda_i(\hat h), \ \   \text{for all}\  i\in\{ 1, 2, ... , n \}.
\end{equation*}
Hence,
\begin{equation*}
\begin{aligned}
 \Phi(\hat h+\tau \chi_{P_m} ) -  \Phi(\hat h)=
\max_{s=0,\ldots,d(\tilde q)}  \sum_{j=1}^{n-s} \lambda_j(\hat h +\tau \chi_{P_m})
- \max_{s=0,\ldots,d(\tilde q)}  \sum_{j=1}^{n-s} \lambda_j(\hat h)
<0,
   \end{aligned}
\end{equation*}
which  implies \eqref{eq:contrad} and then contradicts the minimality of $\hat h$.   Thus, we obtain that  $\hat h=\sum_{m=1}^M  \min(q_{\min}-q_0, \beta_m) \chi_{P_m} (x)$.

Note that, $P_m\subseteq \text{supp}\, \hat h $ if and only if $\beta_m>0$ and,
by Lemma \ref{lemma1}, this holds if and only if $P_m\subseteq D$. Moreover, Lemma \ref{lemma1} implies that
\begin{equation*}
  \hat h=
  \begin{cases}
    q_{\min}-q_0,    &\text{for}\ P_m\subseteq D, \\
    0,   &\text{for}\  P_m\not\subseteq D.
  \end{cases}
\end{equation*}
\end{proof}

\subsection{Stability analysis}
Now, we consider the stability of the proposed monotonicity-based regularization method.  Let the measured noisy data $\Lambda^{\delta}(q)$ satisfy
\begin{equation*}
  \|\Lambda^{\delta}(q)-\Lambda(q)\|_{\mathcal{L}(L^2(\partial \Omega))}\leq \delta,
\end{equation*}
where $\delta>0$ represents the noise level.   For simplicity, we define
 \begin{equation*}
  V:=\Lambda(q)-\Lambda(q_0) \quad \text{and} \quad V^{\delta}:=\Lambda^{\delta}(q)-\Lambda(q_0).
\end{equation*}
To ensure that $V^{\delta}$ is  a self-adjoint operator, one can redefine $V^{\delta}$ as
  \begin{equation*}
  V^{\delta}:=\frac{V^{\delta}+(V^{\delta})^*}{2}.
 \end{equation*}
Furthermore, we replace the monotonicity test \eqref{eq:beta} with
\begin{equation}\label{beta-noise}
   \beta_m^{\delta}:=\max \{\alpha\geq 0: V^{\delta} - \alpha \Lambda'(q_0) \chi_{P_m} \geq_{d(\widetilde{q})} -\delta  I \},
\end{equation}
where $I$ is the identity operator.

Before discussing stability, we present the following key lemma.

\begin{lemma}\label{lem:noise}
Assume that $\|\Lambda^{\delta}(q)-\Lambda(q)\|_{\mathcal{L}(L^2(\partial \Omega))} \leq \delta$, then  for every pixel $P_m$, it holds that
$\beta_m\leq \beta_m^{\delta}$ for all $\delta>0$.
\end{lemma}
\begin{proof}
Since $V-V^{\delta}$ is linear, bounded and self-adjoint, for any $g\in L^2(\partial \Omega)$ with $\|g\|_{L^2(\partial \Omega)}\neq 0$, we have
\begin{equation*}
\begin{aligned}
  |\langle g,  (V-V^{\delta}) g\rangle|&=\|g\|^2 \left|\left \langle \frac{g}{\|g\|}, (V-V^{\delta})\frac{g}{\|g\|}   \right\rangle\right|\\
                                                               & = \|g\|^2 \|V-V^{\delta}\|_{\mathcal{L}(L^2(\partial \Omega))}\\
                                                               & = \|g\|^2 \|\Lambda(q)-\Lambda^{\delta}(q)\|_{\mathcal{L}(L^2(\partial \Omega))}\\
                                                               &\leq \delta \|g\|^2.
\end{aligned}
\end{equation*}
Thus, $ -\delta  I \leq V^{\delta}-V \leq \delta  I$ in quadratic sense.  Furthermore,  according to  the definition of $\beta_m$  \eqref{eq:beta},   one can find that
\begin{equation}\label{eq:V-delta}
V^{\delta}- \alpha \Lambda'(q_0) \chi_{P_m}\geq  V- \alpha \Lambda'(q_0) \chi_{P_m} +(V^{\delta}-V)  \geq_{d(\widetilde{q})} -\delta I, \quad \text{for all} \ \alpha\in[0, \, \beta_k].
\end{equation}
Hence, according to definition of  $\beta_m^{\delta}$ in \eqref{beta-noise}, one can find that $\beta_m\leq \beta_m^{\delta}$ for all $\delta>0$.
\end{proof}

\begin{remark}
As a consequence of Theorem \ref{thm:main} and  Lemma \ref{lem:noise},  it holds that
\begin{description}
  \item[(1)]  If $P_m$ lies inside $D$, then $\beta_m^{\delta}\geq q_{\min}-q_0$.
  \item[(2)]  If  $\beta_m^{\delta}=0$, then  $P_m$ does not lies inside $D$.
\end{description}
\end{remark}

Now we present the following stability result.

\begin{theorem}\label{thm:stability}
Consider the minimization problem
\begin{equation}\label{stability}
  \min_{h\in \mathcal{A}} \left(  \sum_{\{j:\, \lambda_j>0\}} \lambda_j(\bm R^{\delta}(h)) +\delta \|\bm R^{\delta}(h)\|_F \right),
 \end{equation}
where
$ \bm{R}^{\delta}(h):= \mathcal{F}^{\delta}(q)-\mathcal{F}(q_0)-\mathcal{F}'(q_0)h$, $\|\cdot \|_F$ denotes the Frobenius  norm,
and the admissible set corresponding to noisy data  is defined by
  \begin{equation*}
  \mathcal{A}^{\delta}:=\left\{h\in L^{\infty}(\Omega): \ h=\sum_{m=1}^M a_m \chi_{P_m} (x),  \ 0\leq a_m\leq \min(q_{\min}-q_0, \beta_m^{\delta})   \right\}.
\end{equation*}
Then it holds that
\begin{description}
  \item[(i)] Problem \eqref{stability} admits a minimizer;
  \item[(ii)] Let $\hat h=\sum_{m=1}^M \min(q_{\min}-q_0, \beta_m) \chi_{P_m}$  and $\hat h^{\delta}=\sum_{m=1}^M \hat a_m^{\delta} \chi_{P_m}$ be the minimizers of problem \eqref{problem} and \eqref{stability}, respectively. Then $\hat h^{\delta}$ converges pointwise and uniformly to $\hat h$ as $\delta\to 0$.
\end{description}
\end{theorem}
\begin{proof}
(i)  Let
\begin{equation*}
  \Phi^{\delta}(h):=
 \sum_{\{j:\, \lambda_j>0\}} \lambda_j(\bm R^{\delta}(h)) +\delta \|\bm R^{\delta}(h)\|_F.
\end{equation*}
Then, with the same arguments as in the proof of Theorem~\ref{thm:main},
it is obvious that  problem \eqref{stability} admits at least one  minimizer in the compact set $ \mathcal{A}^{\delta}$ since $h\mapsto \Phi^{\delta}(h)$ is continuous.

(ii) Step 1:  convergence of a subsequence of $\hat h^{\delta}$.
For any fixed $m$,  noting  that $0\leq \hat a_m^{\delta}\leq q_{\min}-q_0$, by Weierstrass' theorem, there exists a subsequence $(\hat a_1^{\delta_n}, ...,  \hat a_M^{\delta_n})$ converging to some limit $(a_1, ... , a_M)$, as  $\delta_n\to 0$. It is clear to find that $0\leq a_m\leq q_{\min}-q_0$ for all $m=1, ... ,M$.

Step 2: upper bound and limit.
We shall check that $a_m\leq \beta_m$ for all $m=1, ...  ,M$.  Due to $V^{\delta} \to V$ as $\delta \to 0$, we have
\begin{equation*}
  V-a_m \Lambda'(q_0) \chi_{P_m}=\lim_{\delta_n \to 0}\left (  V^{\delta_n}-\hat a_m^{\delta_n} \Lambda'(q_0) \chi_{P_m} \right)
\end{equation*}
any fixed $m$.  Therefore,  using \eqref{eq:V-delta},  we have
\begin{equation*}
\begin{aligned}
\langle g,  (V-a_m \Lambda'(q_0) \chi_{P_m}) g\rangle
            &=\lim_{\delta_n \to 0} \langle g,   (V^{\delta_n}-\hat a_m^{\delta_n} \Lambda'(q_0) \chi_{P_m} )g \rangle\\
            &\geq_{d(\widetilde{q})}  \lim_{\delta_n \to 0} \langle g,  -\delta_n g \rangle=0,
 \end{aligned}
\end{equation*}
for any $g\in L^2(\partial \Omega)$ with $\|g\|_{L^2(\partial \Omega)}\neq 0$.

Step 3: minimality of the limit.
By Lemma \ref{lem:noise},  one can get that $\min(q_{\min}-q_0, \beta_m)\leq \min (q_{\min}-q_0, \beta_m^{\delta})$ for all $m=1, ... , M $.
Thus, $\hat h$ belongs to the admissible class of the minimization problem \eqref{stability} for all $\delta>0$. By the minimality of $\hat h^{\delta}$,  one can obtain
\begin{equation}\label{eq:R}
\Phi^{\delta}(\hat h^{\delta}) \leq \Phi^{\delta}(\hat h).
\end{equation}
Denote by  $h=\sum_{m=1}^M a_m \chi_{P_m}$, where $a_m$'s are the limits obtained in Step 1.  Let
\begin{equation*}
  \bm{R}^{\delta_n}(\hat h^{\delta_n})=  \left\langle g_i,   \left(V^{\delta_n}- \sum_{m=1}^M \hat a_m^{\delta_n} \Lambda'(q_0) \chi_{P_m}\right)g_j \right\rangle
\end{equation*}
and
\begin{equation*}
 \bm{R}(h)=  \left\langle g_i,   \left(V- \sum_{m=1}^M a_m \Lambda'(q_0) \chi_{P_m}\right)g_j \right\rangle
\end{equation*}
Since $V^{\delta_n}\to V$ and $\hat a_m^{\delta_n} \to a_m$  for any $m\in \{1,...,M\}$ as $\delta_n\to 0$,   it is easy to check that $ \bm{R}^{\delta_n}(\hat h^{\delta_n})\to \bm{R}(h)$ as $\delta_n\to 0$. Thus, we have  $\Phi^{\delta_n}(\hat h^{\delta_n}) \to \Phi(h)$ as $\delta_n\to 0$. Moreover, by a similar argument, one can obtain
$ \Phi^{\delta_n}(\hat h) \to \Phi(\hat h)$  as $\delta_n\to 0$. Thus,  it follows from \eqref{eq:R}
\begin{equation*}
\Phi(h)  \leq \Phi (\hat h).
\end{equation*}
Notice that $\hat h$ is a minimizer of problem \eqref{problem} and $h=\sum_{m=1}^M a_m \chi_{P_m}$ belongs to the admissible class of problem \eqref{problem},  together with the above inequality,  it implies that $h$ is also  a minimizer.  By the uniqueness of the minimizer of problem \eqref{problem} in Theorem \ref{thm:main}, we get $h=\hat h$.

Step 4: convergence of the whole sequence $\hat h^{\delta}$.
We have proved that every subsequence of $(\hat h_1^{\delta}, ... , \hat h_M^{\delta})$ has a convergent subsubsequence, which implies the convergence of the whole sequence $(\hat h_1^{\delta}, ... , \hat h_M^{\delta})$ to $(\min(q_{\min}-q_0, \beta_1),..., \min(q_{\min}-q_0, \beta_M)) $.
\end{proof}

\begin{remark}
In Theorem \ref{thm:stability}, we incorporate the penalty term $\delta \|\bm R^{\delta}(h)\|_F$  into the residual functional. The primary purpose of this addition is to ensure that the minimization problem with noisy data, given in \eqref{stability}, converges to the original noise-free problem \eqref{problem} stated in Theorem \ref{thm:main}.  This formulation is designed to guarantee both theoretical global convergence and robustness to noise in numerical implementations.
\end{remark}

\section{Computational  experiments}

In this section, we present several numerical experiments to demonstrate the effectiveness and efficiency of the monotonicity-based regularization method. Before proceeding, we introduce  some relevant notations and the forward problem calculations.

\subsection{Data generation}
To begin with,  we specify the procedure for generating the artificial data   $\bm V:=\mathcal{F}_n (q)-\mathcal{F}_n(q_0)$ and $\bm S_m:=\mathcal{F}'_n(q_0)\chi_{P_m}$, where $\mathcal{F}_n (q)$ and $\mathcal{F}'_n(q_0)$ are defined in \eqref{eq:Fn} and \eqref{eq:Fn-derivative},  respectively.
Let $\Omega$  be the unit disk centered at the origin in two dimensions.
We consider the Neumann data $g_j$ in the following orthonormal set of $L^2(\partial \Omega)$:
\begin{equation*}
 \left \{1, \ \frac{1}{\sqrt{\pi}}\sin( j \varphi), \  \frac{1}{\sqrt{\pi}}\cos( j \varphi): j=1,2, ..., N_1 \right \}, \quad N=2N_1+1.
\end{equation*}
In the following, we set the number of basis functions to  $N_1=16$, i.e., $N=33$. We note that
\begin{equation*}
  \bm V_{i,j}= \int_{\partial \Omega} g_i \, (\Lambda(q)-\Lambda(q_0)) g_j \,\mathrm{d}s=\int_{\partial \Omega} g_i \, (u_q^j -u_{q_0}^j) \,\mathrm{d}s, \quad i,j=1,2, ... , N,
\end{equation*}
where $u_q^j$ is the solution to the Helmholtz equation \eqref{eq:main} for the refractive index $q$ and the Neumann boundary condition $g_j$.
To reduce computational cost,  we denote by $v^j$ the difference $u_q^j-u_{q_0}^j$,  then $v^j$ satisfies the following system:
\begin{equation}\label{sys:1}
  \begin{aligned}
& \Delta v^j+k^2 q v^j=k^2(q_0-q) u_{q_0}^j , \quad \ \, \text{in}\   \Omega, \medskip\\
&\partial_{\nu}v^j=0,  \qquad \qquad   \qquad \qquad \quad \quad \text{on}\ \partial \Omega,
\end{aligned}
\end{equation}
where $u_{q_0}^j$ admits a unique solution on the unit disk:
\begin{equation*}
  u_{q_0}^j(r, \varphi)=
\begin{cases}
&\displaystyle\frac{J_{m}(kr\sqrt{q_0}  )}{k\sqrt{q_0\pi} J'_{m}(k\sqrt{q_0} )}  \sin( j \varphi), \quad \text{if}\  g_j=\frac{1}{\sqrt{\pi}}\sin( j \varphi), \medskip\\
&\displaystyle \frac{J_{m}(kr\sqrt{q_0} )}{k\sqrt{q_0\pi} J'_{m}(k\sqrt{q_0} )} \cos( j \varphi), \quad \text{if}\  g_j=\frac{1}{\sqrt{\pi}}\cos( j \varphi).
\end{cases}
\end{equation*}
Here the pair $(r, \varphi)$ represents the polar coordinates for the unit disk $\Omega$.  Furthermore,  the  system given in \eqref{sys:1}  is  solved  using the COMSOL (a finite element software) via the coefficient form PDE model. Interested readers could refer to the EIT model presented in \cite{Harrach21JDMV} for more details.

Thus, each element of the matrix $\bm V$ can be represented by
  \begin{equation*}
      \bm V_{i,j}=\int_{\partial \Omega} g_i \, v^j \,\mathrm{d}s, \quad i,j=1,2, ... , N.
  \end{equation*}
  Moreover,  each entry of the matrix $\bm S_m$, i.e.,
\begin{equation*}
  \bm S_m^{i,j}=\int_{\partial \Omega}  g_i (\Lambda' (q)\chi_{P_m}) g_j \, \mathrm{d}s=\int_{P_m} k^2  u_{q_0}^i  u_{q_0}^j \,\mathrm{d}x,  \quad i,j=1,2, ... , N,
\end{equation*}
can be computed on each triangular element $P_m$  by using area coordinates  with the linear Lagrange elements.
To test the stability of the proposed method, we also add some noise to the measurements, i.e.,
\begin{equation*}
  \bm V^{\delta}:=\bm V+ \frac{\bm E}{\|\bm E \|_F}\delta,
\end{equation*}
where $\delta>0$ represents the noise level and $\bm E$  is a random matrix in  $\mathbb{R}^{N\times N}$ with a uniform distribution between  $-1$ and $1$.  To ensure that $\bm V^{\delta}$ is a symmetric matrix, we  redefine it as $\bm V^{\delta}=\left((\bm V^{\delta})^*+\bm V^{\delta}\right)/2$.

Next, we provide the details for determining $d(\widetilde{q})$.
Noting that $d(\widetilde{q})$  is the number of negative Neumann eigenvalues of the operator  $-(\Delta+k^2 \widetilde{q})$,
we introduce the variational form of the equation \eqref{eq:modified-Helmholtz}, given by
\begin{equation*}
  \int_{\Omega} (\nabla u\cdot \nabla v-k^2 \widetilde{q}uv )\, \mathrm{d}x=0, \quad \text{for\,all}\ v\in H^1(\Omega).
\end{equation*}
For simplicity, we set  $\Omega$ to be  the unit disk,  and let $\Omega_0$ be a concentric disk with radius  $r_0<1$.  Moreover, let   $\widetilde{q}$ be the piecewise constant function defined in \eqref{eq:q-tilde}.
We then utilize the linear Lagrange elements to discretize the above variational equation to obtain the algebraic system
\begin{equation*}
(\bm K-k^2 \bm M) \bm u= \bm 0,
\end{equation*}
where $\bm K$  and $\bm M$ denote the stiffness and mass matrices, respectively, and $\bm u$ is the vector of unknown coefficients.
Thus, $d(\widetilde{q})$ is obtained as the count of negative eigenvalues of the discretized matrix  $\bm K-k^2 \bm M$. In Figure \ref{fig:negative-eigenvalues}, we present numerical results for computing  $d(\widetilde{q})$ for  varying  radius $r_0$  with different wavenumber $k$ and refractive index $q_{\max}$. It is clear to see that the number of negative eigenvalues decreases as the radius $r_0$ decreases.  This observation indicates that a more accurate upper bound for $d(q)$ can be obtained when {\it a priori} information about the compact support of $q$ is available.

\begin{figure}
\hfill\subfigure[$k=1,  \, q_{\max}=9 $]{\includegraphics[width=0.48\textwidth]
                   {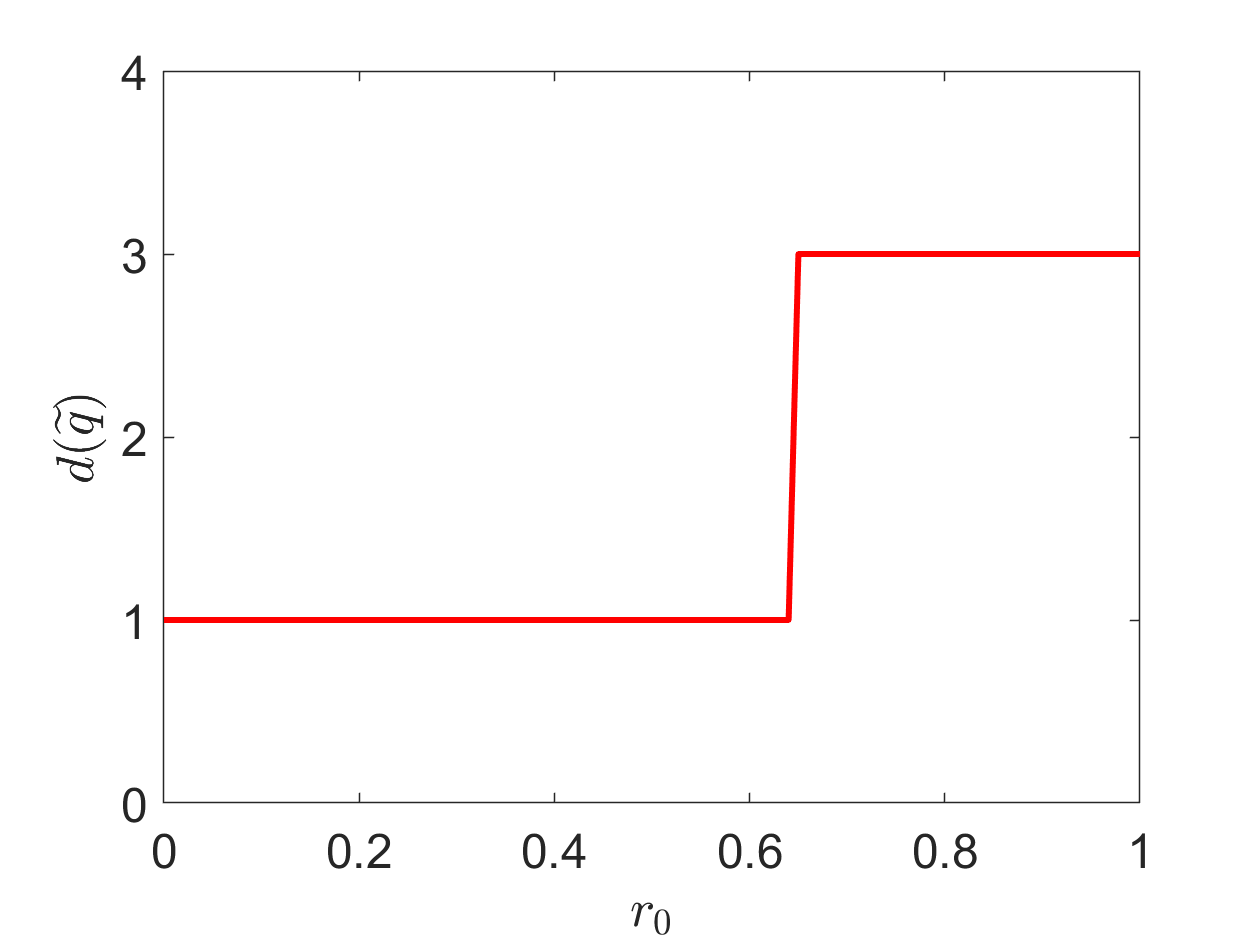}}\hfill
\hfill\subfigure[$k=3, \, q_{\max}=5$]{\includegraphics[width=0.48\textwidth]
                   {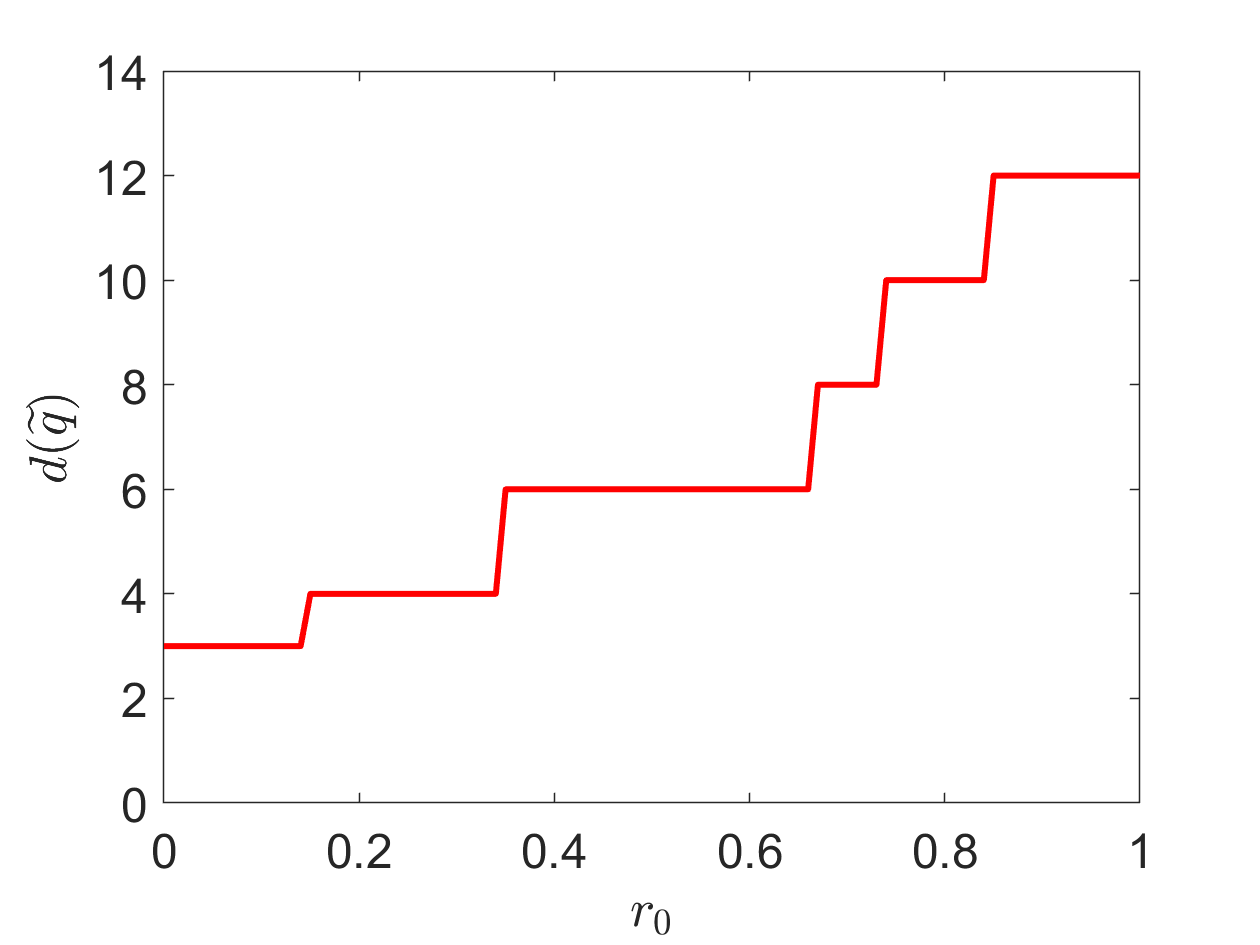}}\hfill
\caption{\label{fig:negative-eigenvalues}  The number of negative eigenvalues $d(\widetilde{q})$  for varying radius $r_0$, with different wavenumber $k$ and refractive index $q_{\max}$.}
\end{figure}

Finally, we introduce a strategy to determine  $\beta_m^{\delta}$ corresponding to the pixel  $P_m$. By replacing the infinite-dimensional operators  $V^{\delta}$ and  $\Lambda'(q_0) \chi_{P_m}$ with  the $n\times n$ matrices $\bm V^{\delta}$ and $\bm S_m$,
the monotonicity test in finite dimensions can be rewritten as:
\begin{equation}\label{eq:beta-m-n}
  \beta_m^{\delta, n}=\max\{ \alpha\geq 0:\   \bm V^{\delta} -\alpha \bm S_m \geq_{d(\widetilde{q})} -\delta \bm I\}.
\end{equation}
As before, let $\lambda_1(A)\geq \dots \geq \lambda_n(A)$ denote the eigenvalues of a symmetric matrix $A\in \mathbb{R}^{n\times n}$ in descending order.
Note that $\lambda_n(\bm S_m)>0$ since $\bm S_m=\mathcal{F}'_n(q_0)\chi_{P_m}$ is a positive-definite matrix by Lemma \ref{lemma2}. Hence, we can define
\begin{equation*}
\alpha_0:=\left\{ \begin{array}{l l} \frac{1-\lambda_n(\bm V^{\delta}+\delta \bm I)}{\lambda_n(\bm S_m)} & \text{ if $\lambda_n(\bm V^{\delta}+\delta \bm I)\leq 0$}\\
0 & \text{ if $\lambda_n(\bm V^{\delta}+\delta \bm I)> 0$}.
\end{array}\right.
\end{equation*}
Then $\bm V^{\delta}+\delta \bm I+\alpha_0 \bm S_m$ is positive definite and thus possesses a Cholesky decomposition
\begin{equation}\label{eq:decomposition}
\bm V^{\delta}+\delta \bm I+\alpha_0 \bm S_m=\bm L \bm L^*,
\end{equation}
where $\bm L$ is an invertible lower triangular matrix with real and positive diagonal entries. Thus, for each
$\alpha\in[0, \beta_m^{\delta, n}]$, it follows that
\begin{equation*}
\bm V^{\delta} -\alpha \bm S_m \geq_{d(\widetilde{q})} -\delta \bm I
\end{equation*}
is equivalent to
\begin{equation*}
  (\alpha+\alpha_0)  \bm L^{-1}\bm S_m  (\bm L^{*})^{-1} \leq_{d(\widetilde{q})}  \bm I,
\end{equation*}
and this is equivalent to
\begin{equation*}
(\alpha+\alpha_0)\lambda_{d(\widetilde{q})+1}(\bm L^{-1}\bm S_m  (\bm L^{*})^{-1})\leq 1.
\end{equation*}
According the definition of $\beta_m^{\delta, n}$ in  \eqref{eq:beta-m-n}, the maximal value of $\alpha$ that satisfies the last formula is given by
\begin{equation*}
  \beta_m^{\delta, n}= \frac{1}{\lambda_{d(\widetilde{q})+1} (\bm L^{-1}\bm S_m (\bm L^{*})^{-1})}-\alpha_0.
\end{equation*}

\subsection{Numerical inversion} Based on the above discussions, the  residual matrix with noise  $\bm R^{\delta}(h)$ can be expressed as
\begin{equation*}
  \bm R^{\delta}(h)=\bm V^{\delta}-\sum_{m=1}^M a_m \bm S_m.
\end{equation*}
In what follows, inspired by \cite{Harrach22OL}, we employ the semidefinite programming (SDP) to solve the  minimization problem in  \eqref{stability}.   Note that the sum of all positive eigenvalues of the matrix $\bm R(h)$ is equal to the optimal value of the following SDP:
\begin{equation*}
   \sum_{\{j:\, \lambda_j>0\}} \lambda_j(\bm R^{\delta}(h)) = \min_{\substack{\bm X \geq 0\\ \bm X - \bm R^{\delta}\geq 0}}  \text{trace}(\bm X).
\end{equation*}
Therefore, the minimization problem
\begin{equation*}
  \min_{h\in \mathcal{A}} \left(  \sum_{\{j:\, \lambda_j>0\}} \lambda_j(\bm R^{\delta}(h)) +\delta \|\bm R^{\delta}(h)\|_F \right),
 \end{equation*}
is equivalent to the following constrained optimization:
\begin{equation}\label{eq:mp}
  \min_{\substack{h\in \mathcal{A}\\\bm X \geq 0\\ \bm X - \bm R^{\delta}\geq 0}} \left(  \text{trace}(\bm X) +\delta \|\bm R^{\delta}(h)\|_F \right).
 \end{equation}
Here  we  use a MATLAB package \textit{CVX}  to solve the above minimization problem.

\begin{remark}
 It is noted  that the sum of $(n-s)$ largest  eigenvalues satisfies:
\begin{equation*}
 \sum_{j=1}^{n-s} \lambda_j(\bm R(h))=\max_{\substack{\bm W\in \mathbb{R}^{n\times (n-s)}\\ \bm W^{\top} \bm W=\bm I_{n-s}}}   \operatorname{trace}( \bm W^{\top}\bm R(h) \bm W).
\end{equation*}
Since the matrix-valued function $\bm R(h)$ depends continuously and linearly on $h$, and $\operatorname{trace}(\bm W^{\top} \bm R(h) \bm W)$ is linear in $h$ for any fixed orthonormal $\bm W$, it follows that the pointwise maximum over such linear functions is convex. Moreover, as shown in \eqref{eq:Max},   the sum of all positive eigenvalues of $\bm R(h)$ can be expressed as the maximum over a finite collection of these convex functions.  Therefore  the function  $h\to \sum_{\{j:\, \lambda_j>0\}} \lambda_j(\bm R^{\delta}(h))$ is convex.  Additionally,  the Frobenius norm $\|\bm R^{\delta}(h)\|_F$  is also convex.  Hence, the full objective function in the equivalent optimization problem   \eqref{eq:mp} is  convex and can be efficiently solved using the convex solver \textit{CVX}.

\end{remark}

  \begin{figure}
\subfigure[Extremely fine mesh for forward problem]{\includegraphics[width=0.45\textwidth]
                   {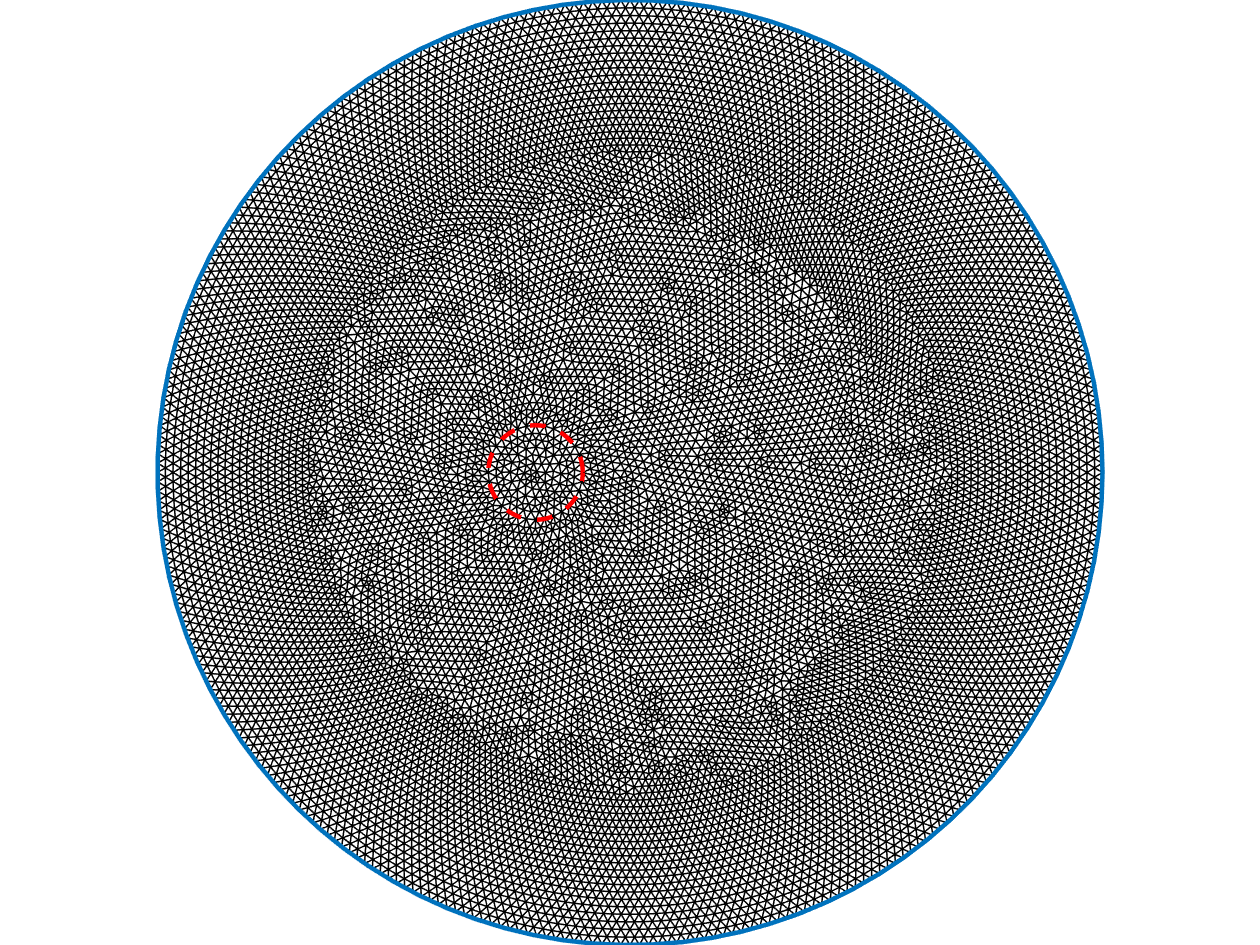}}
\subfigure[Extra fine mesh for  inverse problem]{\includegraphics[width=0.45\textwidth]
                   {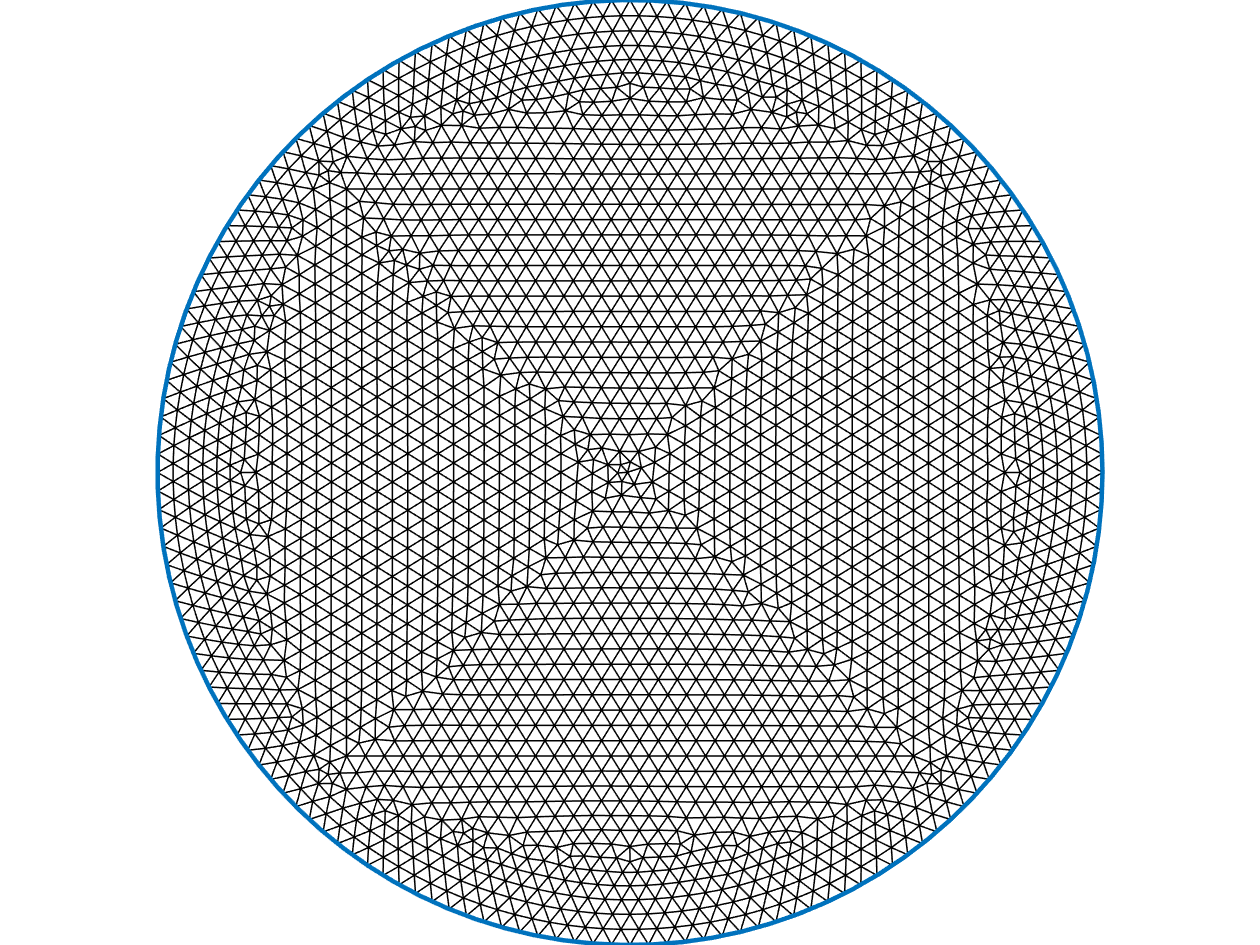}}
\caption{\label{fig:mesh}  Different triangular meshes are used for the forward and inverse problems, respectively.}
\end{figure}

Next we present three examples to verify our theoretical findings. In the numerical experiments, the forward problem is solved using an extremely fine mesh, as illustrated in Figure \ref{fig:mesh}(a). To avoid the so-called inverse crime, the reconstruction domain is independently discretized into 5224 uniformly distributed extra fine triangular elements  $\{P_m\}$, as shown in Figure \ref{fig:mesh}(b). Throughout all experiments, the wavenumber is set to  $k=1$ , and the background refractive index is taken to be  $q_0=1$ in $\Omega\backslash D$.

\begin{figure}
\centering
\subfigure[Exact scatter]{\includegraphics[width=0.44\textwidth]
                   {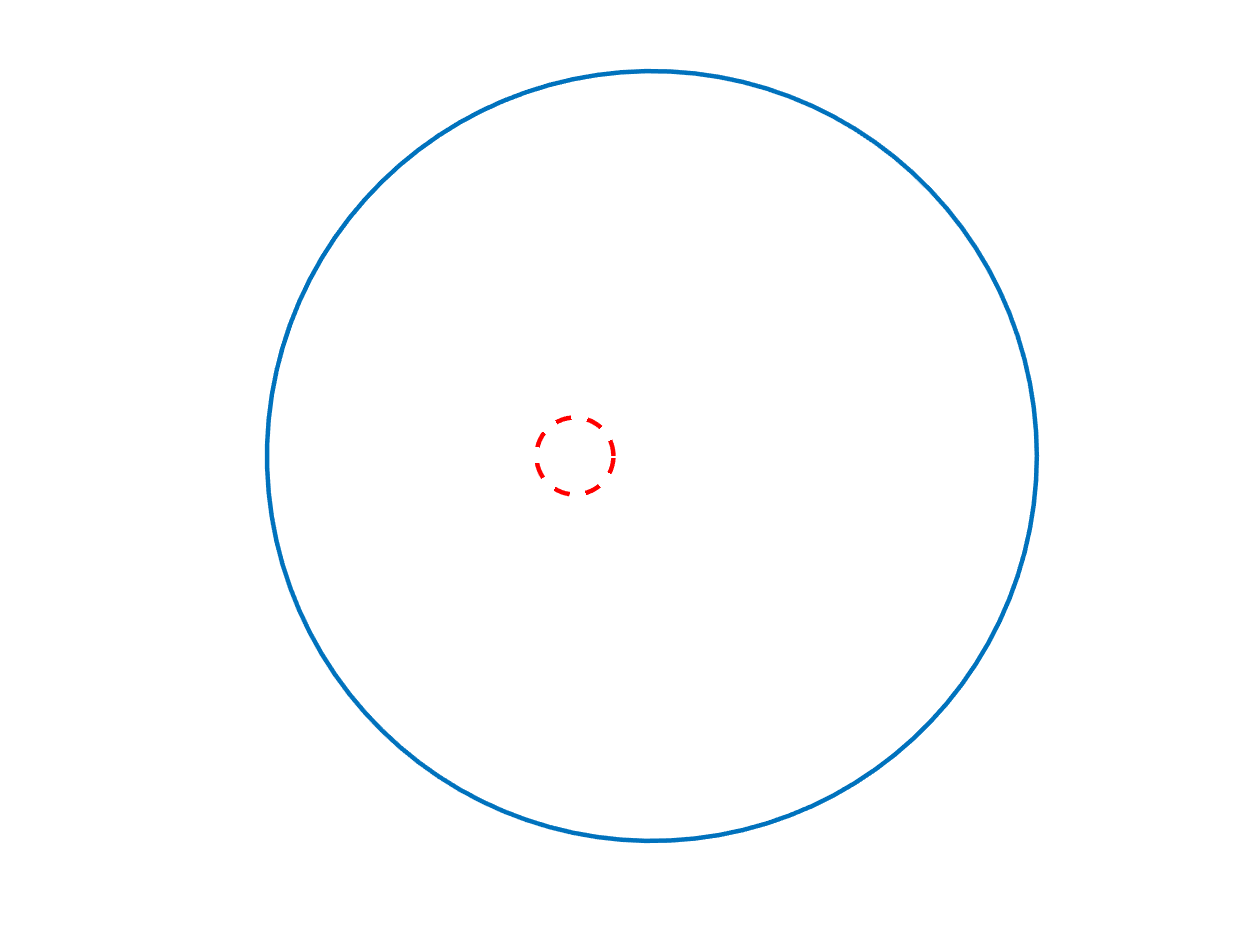}}
\subfigure[$\delta=10^{-6} \%$]{\includegraphics[width=0.45\textwidth]
                   {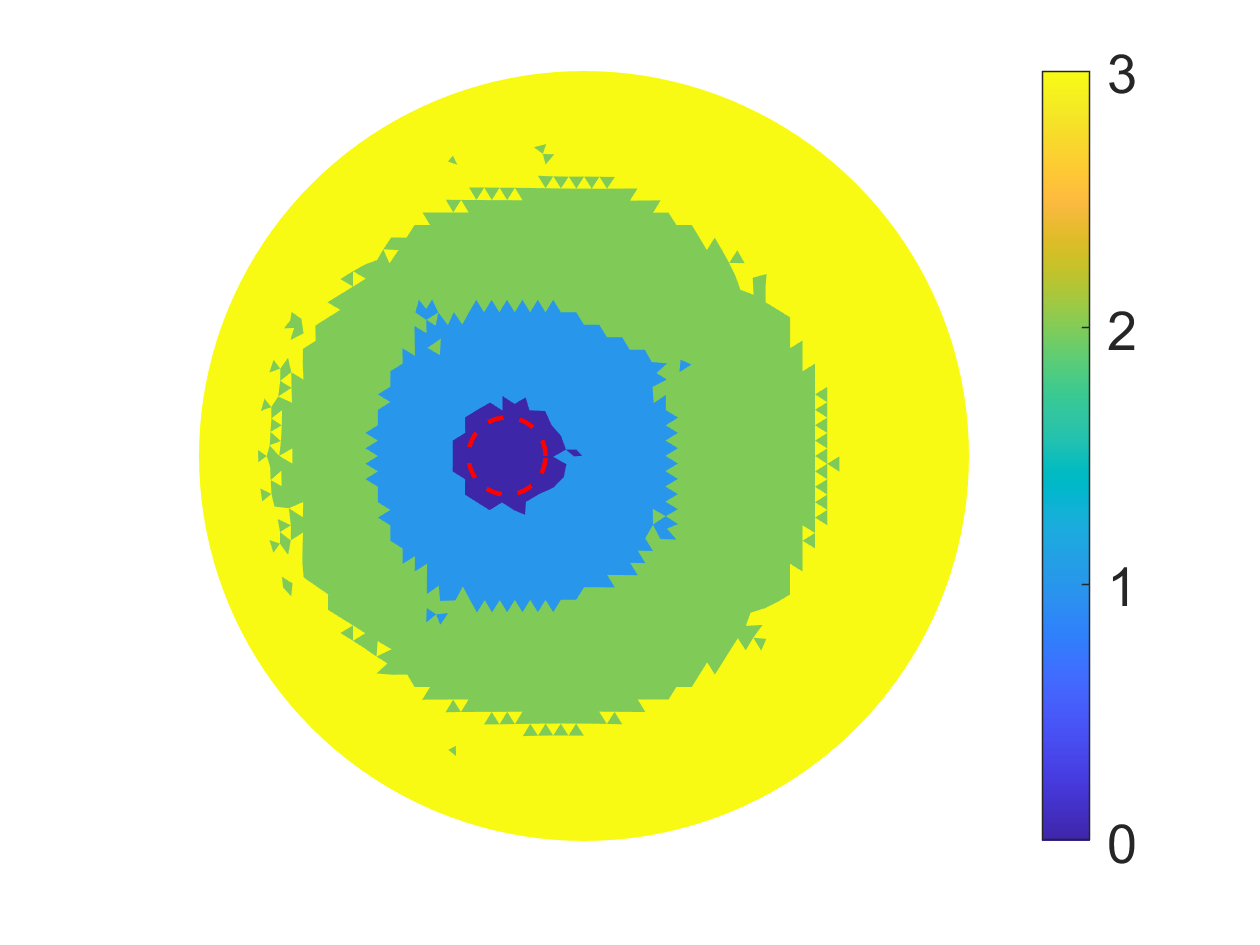}}\\
\subfigure[$\delta=10^{-2} \%$]{\includegraphics[width=0.45\textwidth]
                   {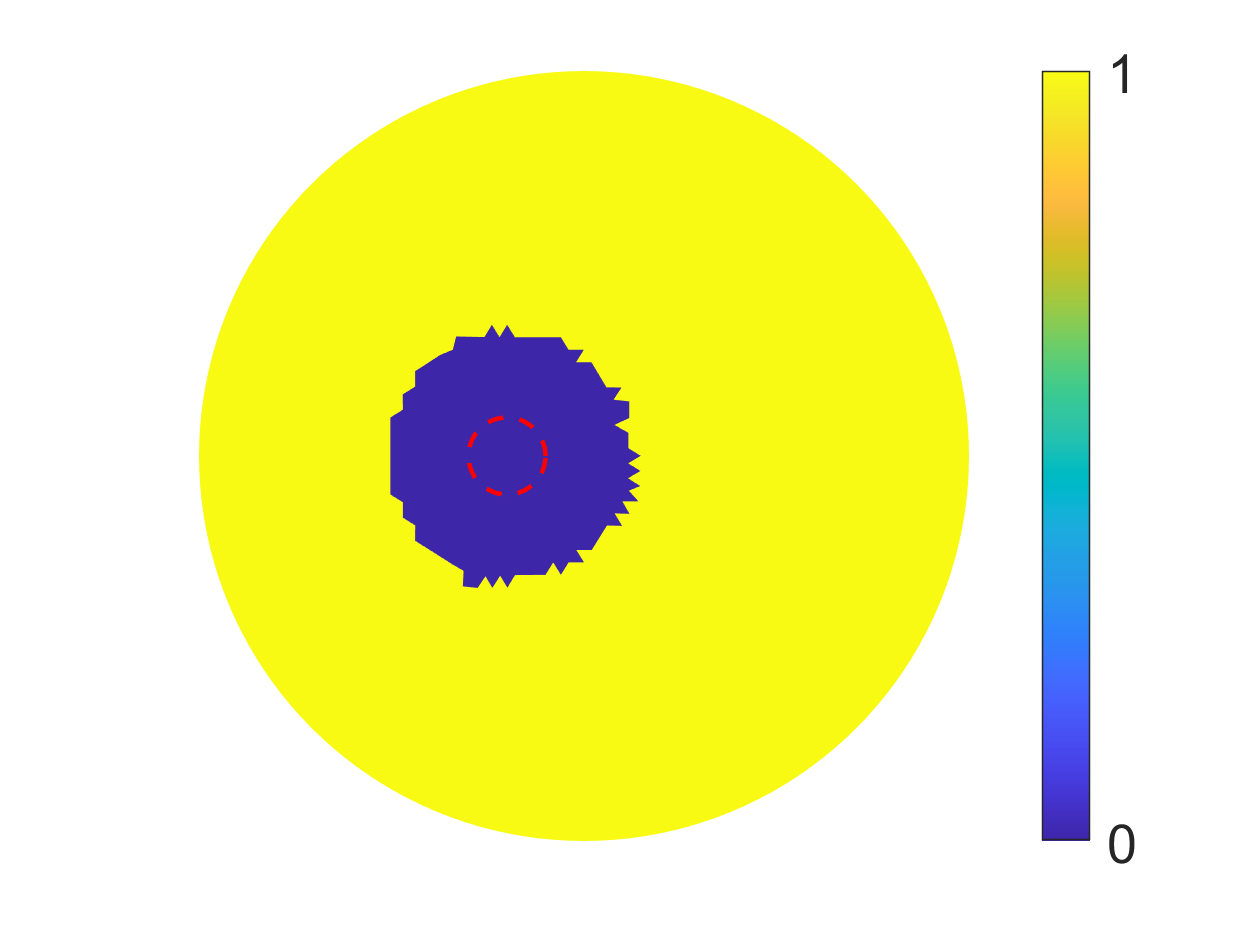}}
 \subfigure[$\delta=10^{-1}  \%$]{\includegraphics[width=0.45\textwidth]
                   {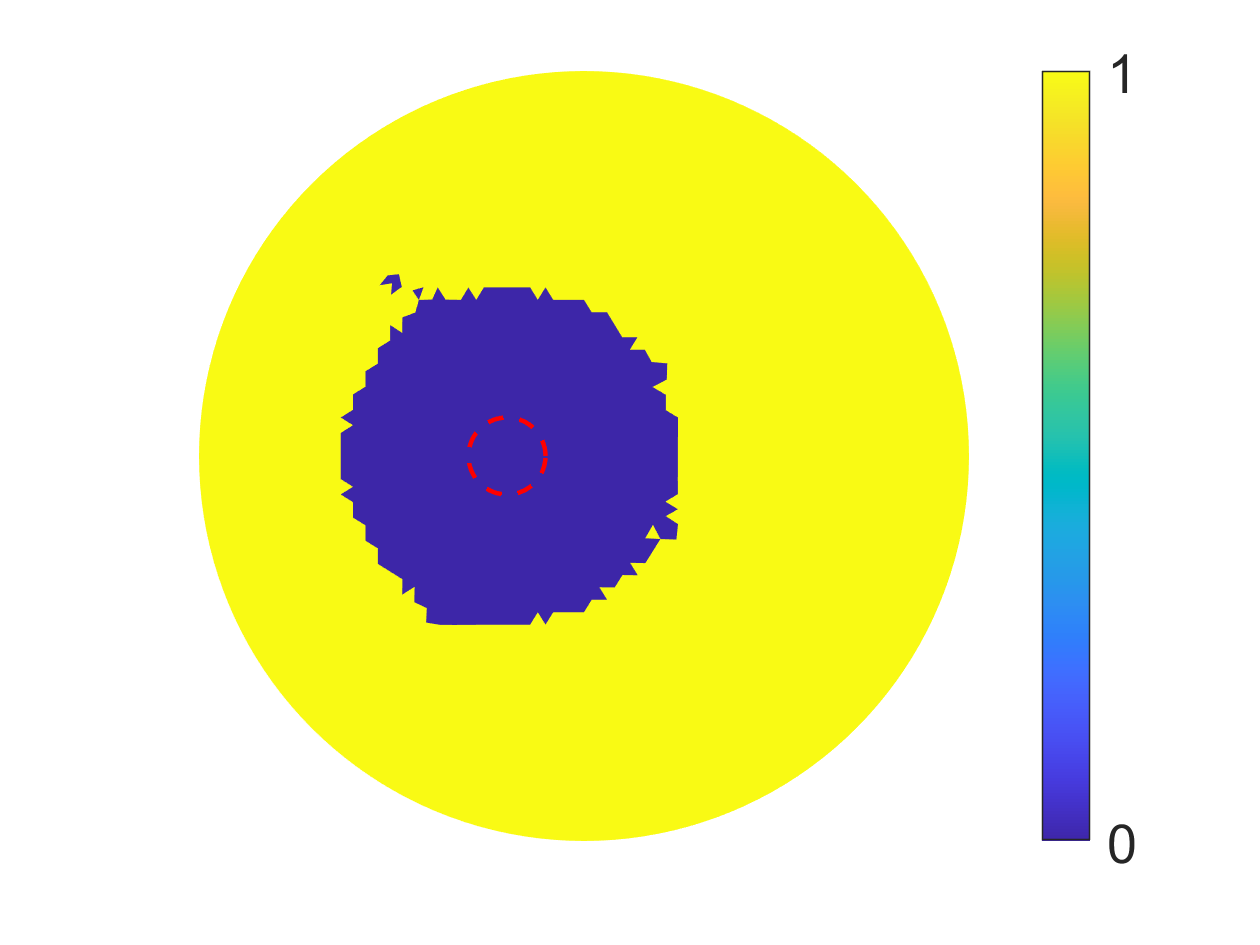}}\\
\caption{\label{fig:single}  Contour plots of the number of the negative eigenvalues for the operator $ \bm V^{\delta} -\alpha \bm S_m +\delta \bm I$  at each pixel $P_m$ under different noise levels $\delta$, where $\alpha=8$. }
\end{figure}

\begin{figure}
\centering
\subfigure[$\delta=1\%$]{\includegraphics[width=0.45\textwidth]
                   {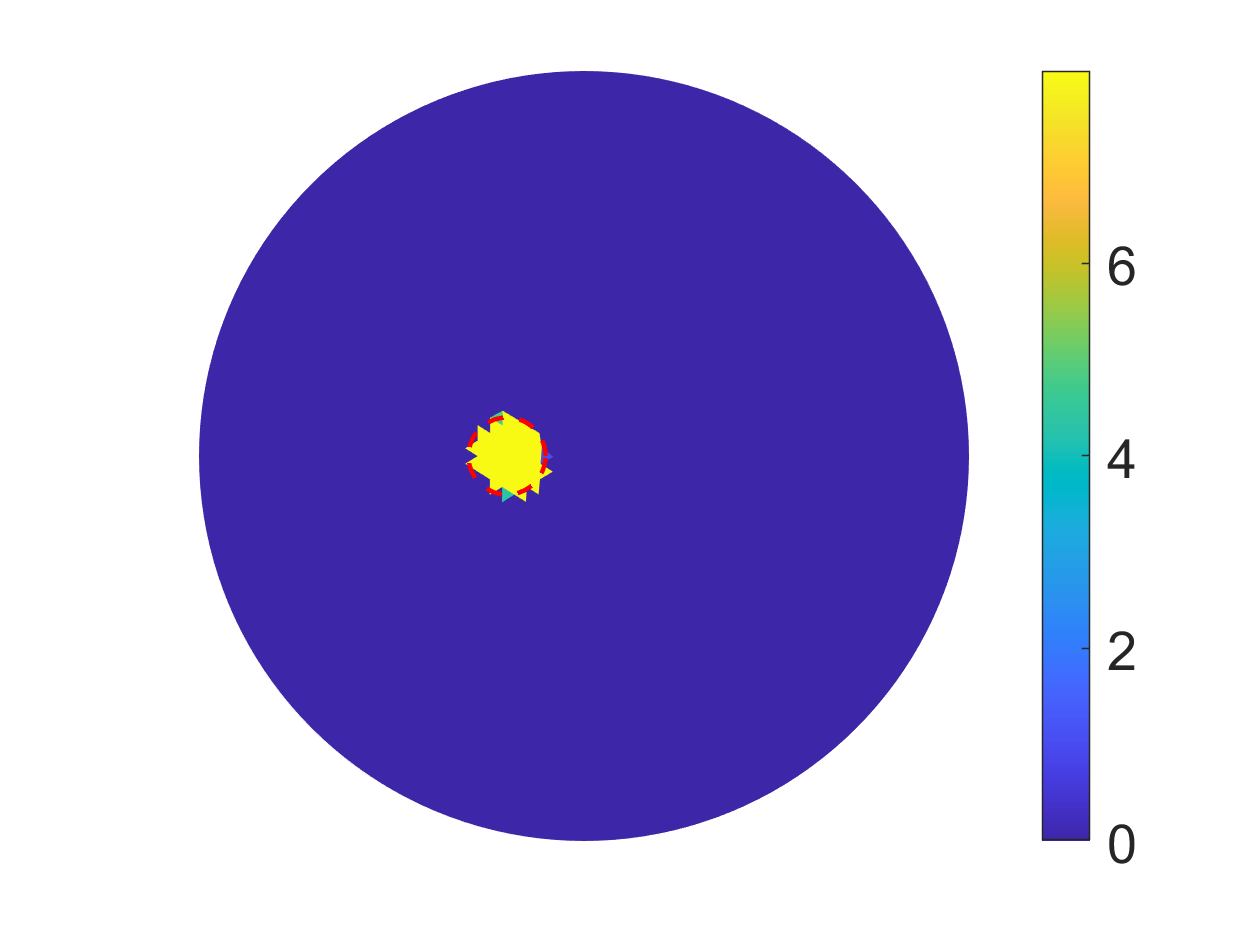}}
\subfigure[$\delta=10\%$]{\includegraphics[width=0.45\textwidth]
                   {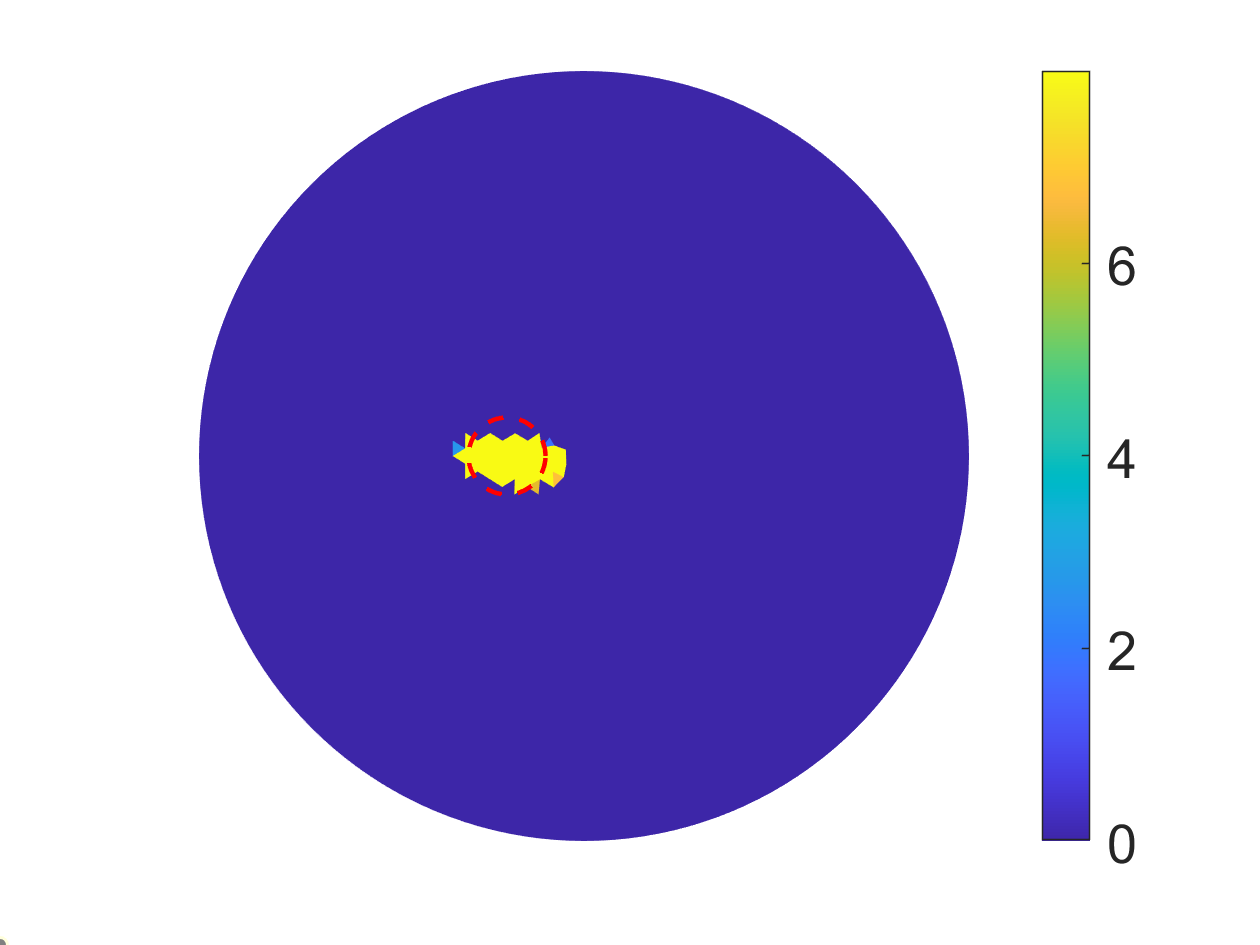}}\\
\caption{\label{fig:single-regularization}  Contour plots  of  the support of minimizer $\hat h^{\delta}$ via the proposed monotonicity-based regularization method with different noise levels $\delta$. }
\end{figure}

\begin{example}
In the first example, we consider a small circular scatterer $D$ centered at $(-0.2, 0)$ with a radius of  $0.1$,  as shown in Figure \ref{fig:single}(a). Here the red dashed curve denotes the shape of the scatterer $D$.  The refractive index within the scatterer is set to $q=9$ in $D$.
Figure \ref{fig:single} presents  the contour plots of  the number of the negative eigenvalues for the operator
$ \bm V^{\delta} -\alpha \bm S_m +\delta \bm I$
evaluated at each pixel $P_m$ under different noise level $\delta$, where the parameter is set to $\alpha=q_{\min-1}=8$. Together with Figure \ref{fig:negative-eigenvalues}(a), these results verify the inequality
\begin{equation*}
\bm V^{\delta} -\alpha \bm S_m \geq_{d(\widetilde{q})} -\delta \bm I
\end{equation*}
 defined in \eqref{eq:beta-m-n}. This  result also implies  that the number of negative eigenvalues inside the scatterer is bounded above by $d(\widetilde{q})=1$.  Furthermore, one can observe that the exact scatterer is located within the compact support of the region corresponding to the number of minimum eigenvalues. However, this region becomes larger as the noise level increases.  Therefore, it is difficult to accurately identify the shape of the scatterer only using  the monotonicity relation under high noise levels. To overcome this challenge, we employ the monotonicity-based regularization method defined in \eqref{eq:mp} to determine the scatterer.  As shown in Figure \ref{fig:single-regularization}, the proposed method exhibits good performance in reconstructing the unknown scatterer under high noise levels, thereby demonstrating its robustness and stability.
\end{example}

\begin{example}
In the second example,  we consider a non-convex,   pear-shaped scatterer $D$ centered at $(0,\, 0)$, see Figure \ref{fig:pear}(a).  The boundary of the scatterer is parameterized by
\begin{equation*}
 (x,\, y)= (0.2+0.03\cos 3t)(\cos t, \ \sin t), \quad t\in[0, \, 2\pi].
\end{equation*}
The refractive index of the scatterer $D$ is set to  $q=5$. Here,  the upper bound of the  number of negative eigenvalues is set to $d(\widetilde{q})=1$.  In order to illustrate the advantages of our approach, we consider the following minimization problem under the Frobenius norm for comparison:
\begin{equation*}\label{eq:dis-mini-fro}
  \min_{0\leq a_m\leq \min(q_{\min}-1, \beta_m^{\delta, n})}  \left\|\bm R^{\delta}(h) \right\|_F.
 \end{equation*}
where  $\beta_m^{\delta, n}$ is defined in  \eqref{eq:beta-m-n}. Figures~\ref{fig:pear}(b) and \ref{fig:pear}(c) show the reconstructions under the Frobenius norm-based method with different noise levels. Moreover, Figures~\ref{fig:pear}(d) and \ref{fig:pear}(e) present the reconstructions based on the sum of  all positive eigenvalues with  the corresponding noise levels.  By comparing these reconstruction results,
it is clear that both methods are robust to noise, even under a 10\% noise level.
However, our method yields more accurate reconstructions than the Frobenius norm-based approach as the noise level $\delta=10^{-9}\%$. This enhancement can be attributed to the global convergence guarantee inherent in our method, whereas the minimization of the residual functional under the Frobenius norm lacks such a guarantee.
\end{example}

\begin{figure}[H]
\centering
\subfigure[ ]{\includegraphics[width=0.45\textwidth]
                   {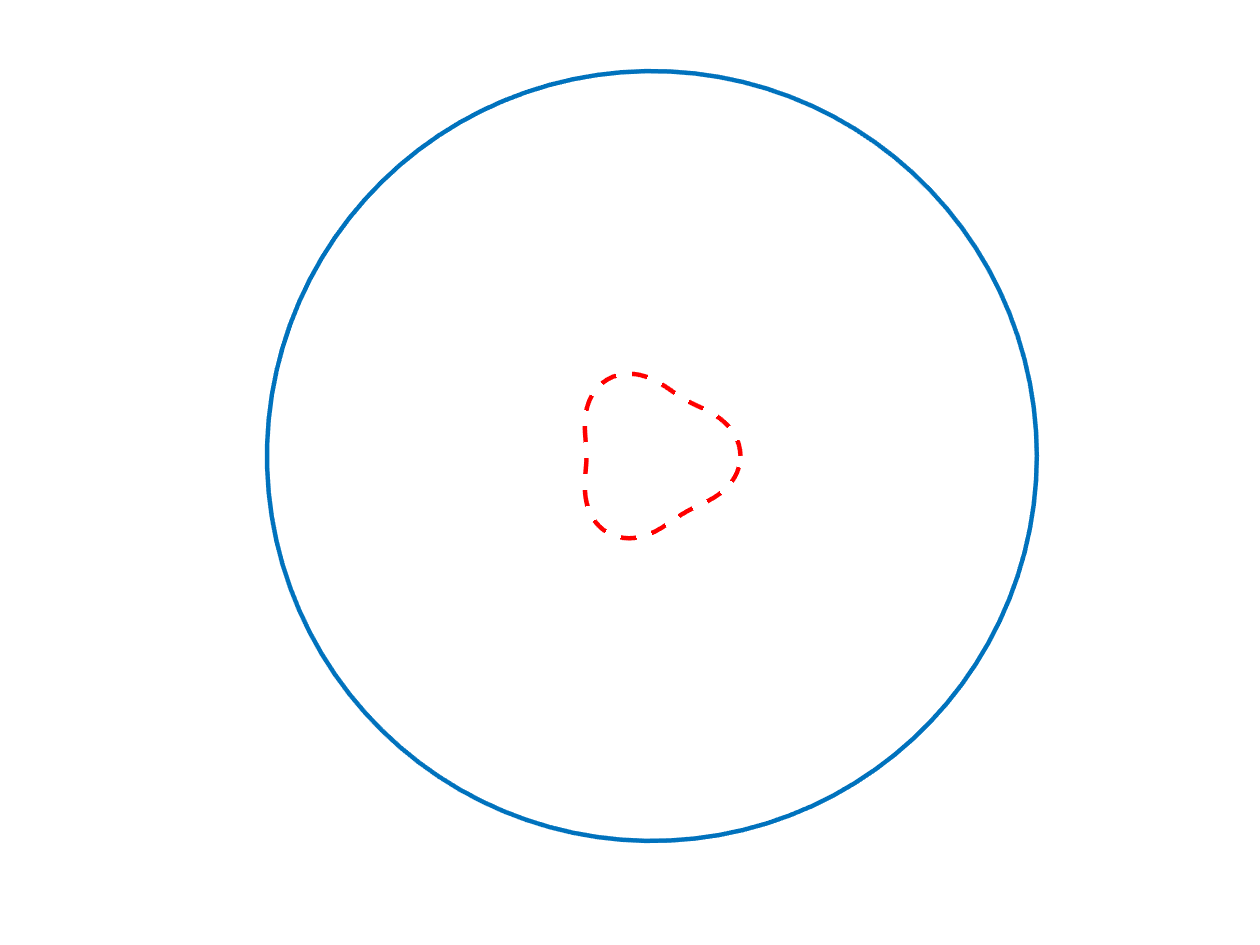}}\\
\subfigure[$\delta=10^{-9}\%$]{\includegraphics[width=0.45\textwidth]
                   {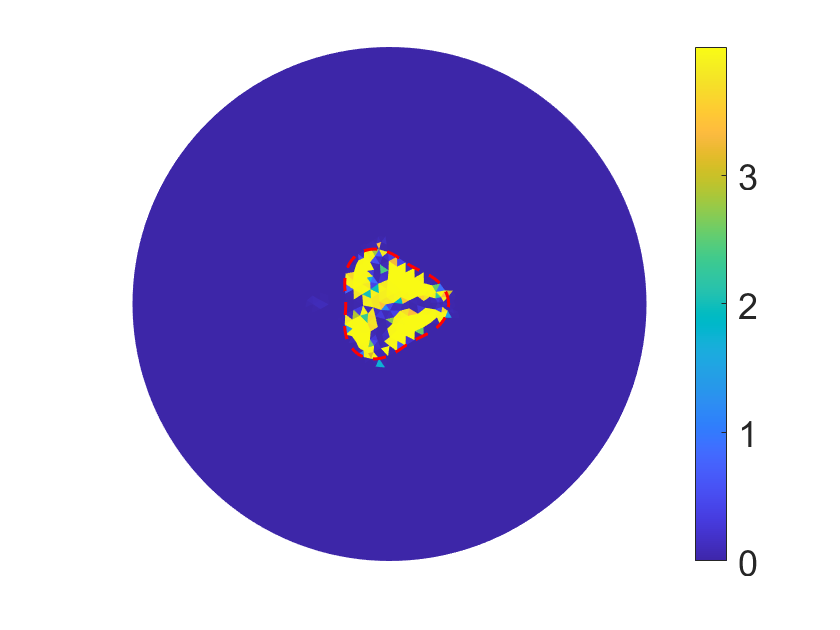}}
\subfigure[$\delta=10\%$]{\includegraphics[width=0.45\textwidth]
                   {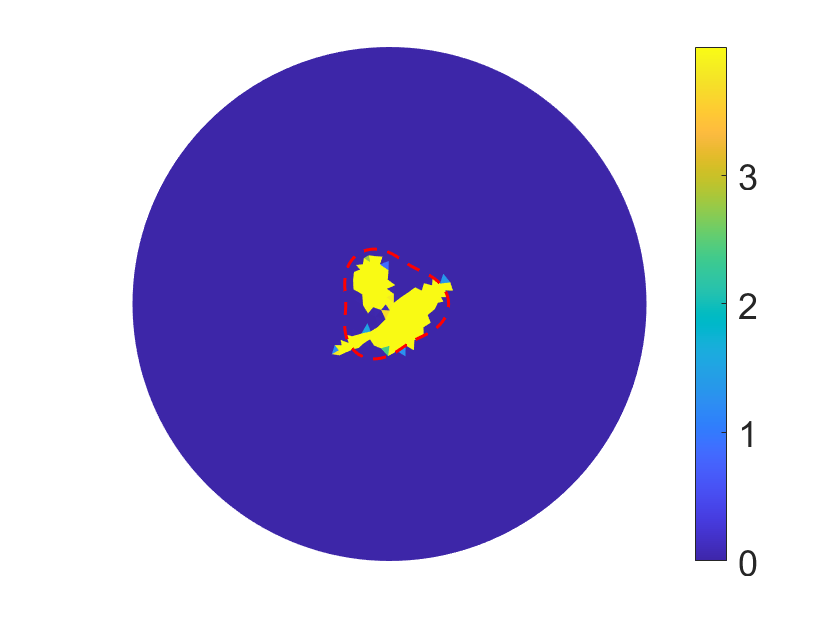}}\\
\subfigure[$\delta=10^{-9}\%$]{\includegraphics[width=0.45\textwidth]
                   {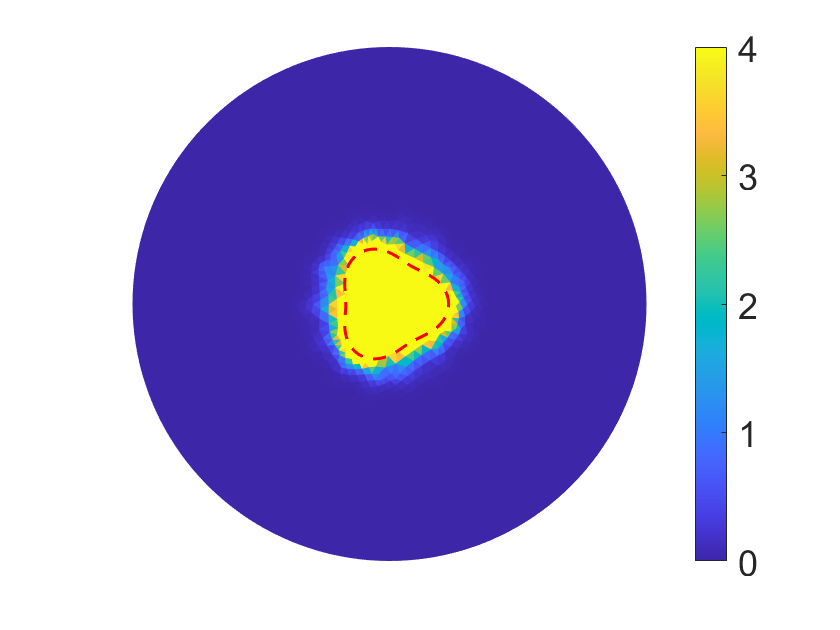}}
\subfigure[$\delta=10\%$]{\includegraphics[width=0.45\textwidth]
                   {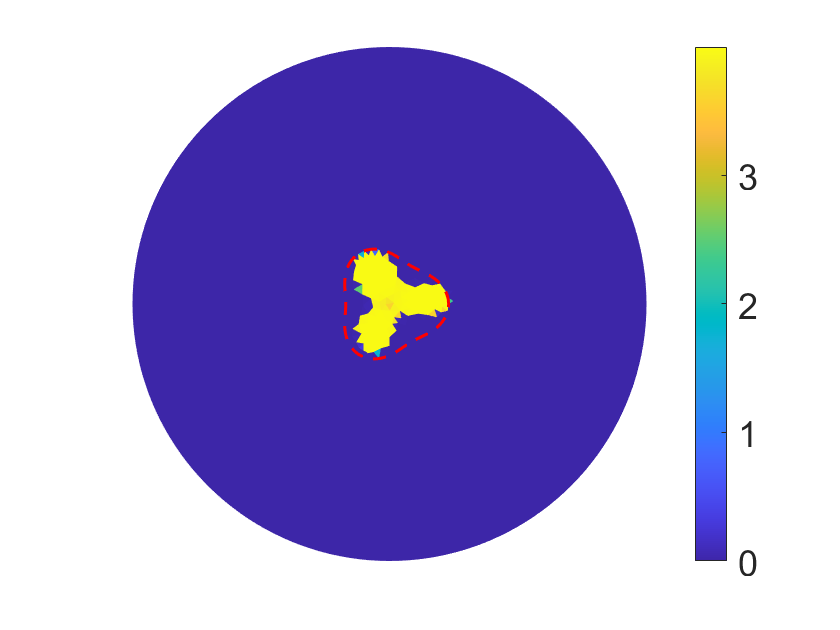}}
\caption{\label{fig:pear} (a) A schematic illustration of the geometric setting;  (b), (c) contour plots of the support of minimizer $\hat h^{\delta}$ under Frobenius norm with different noise levels; (d), (e)  contour plots of the support of minimizer $\hat h^{\delta}$ under the sum of all positive eigenvalues with different noise levels.}
\end{figure}

\begin{example}
In the final example,  we consider a more challenging case.  The scatterer $D$ consists of two small circles with a radius of $0.1$. One is centered at $(-0.35, -0.35)$, and the other at $(0.35, 0.35)$, see Figure \ref{fig:two-scatterers} (a). The refractive index inside $D$ is set to $q=3$. Moreover, the upper bound of the number of negative eigenvalues is chosen as $d(\widetilde{q})=2$.
In this example, we aim to show the necessity of adding the penalty term $\delta \|\bm R^{\delta}(h)\|_F$ to the residual functional. For comparison, we also consider the following minimization problem without the penalty term:
\begin{equation}\label{eq:dis-mini-no-constraint}
 \min_{\substack{h\in \mathcal{A}\\\bm X \geq 0\\ \bm X - \bm R^{\delta}\geq 0}} \text{trace}(\bm X).
 \end{equation}
Figures~\ref{fig:two-scatterers}(b) and \ref{fig:two-scatterers}(c) show the reconstruction results obtained without the penalty term. Although the approach defined in \eqref{eq:dis-mini-no-constraint} performs well when there is almost no noise, it fails to identify the scatterers even under 1\% noise. Moreover,  by comparing Figures~\ref{fig:two-scatterers}(c) and \ref{fig:two-scatterers}(e), one can observe that our method successfully reconstructs multiple scatterers under higher noise levels, which also demonstrates  its robustness and reliability.
\end{example}

\begin{figure}
\centering
\subfigure[ ]{\includegraphics[width=0.45\textwidth]
                   {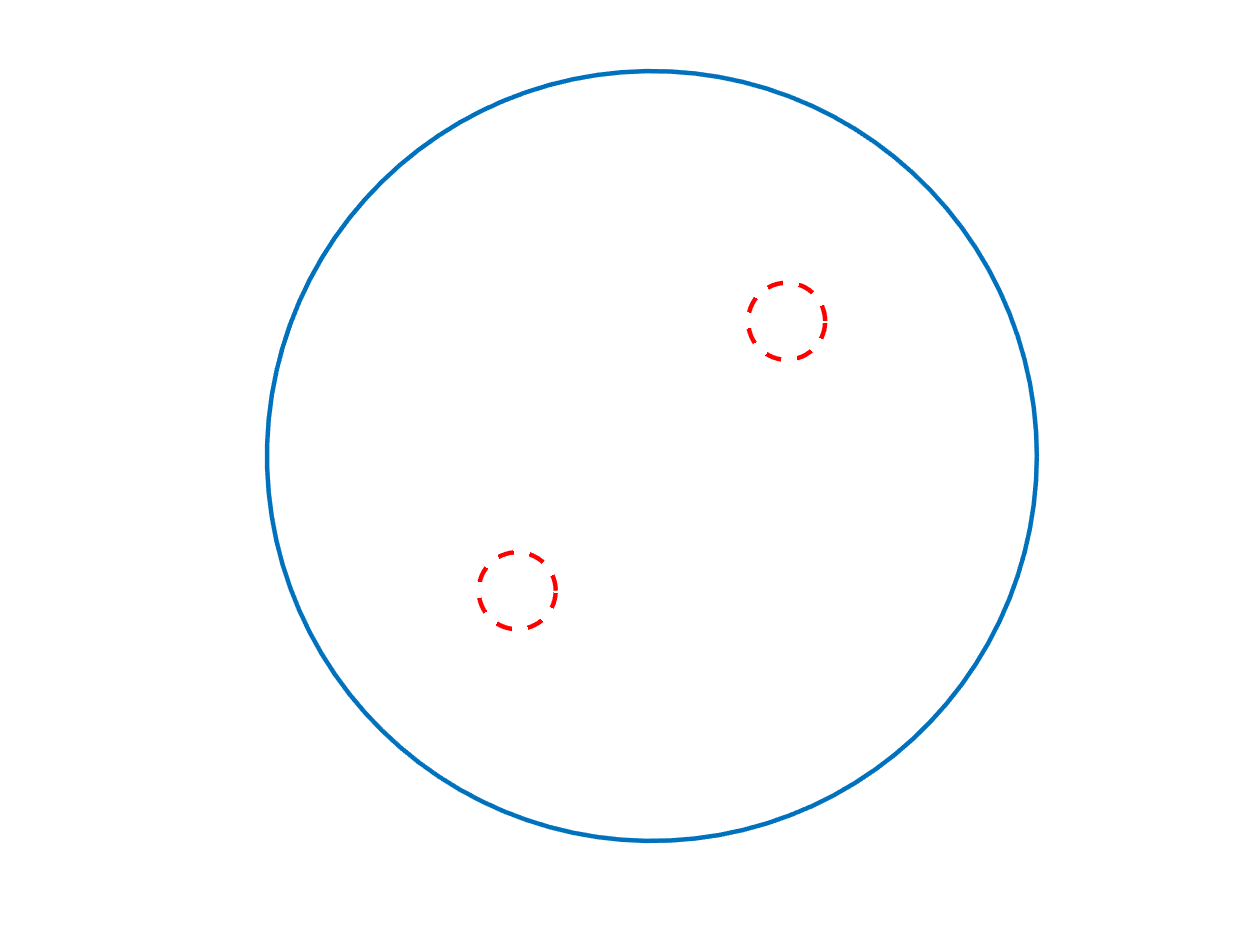}}\\
\subfigure[$\delta=10^{-9}\%$]{\includegraphics[width=0.45\textwidth]
                   {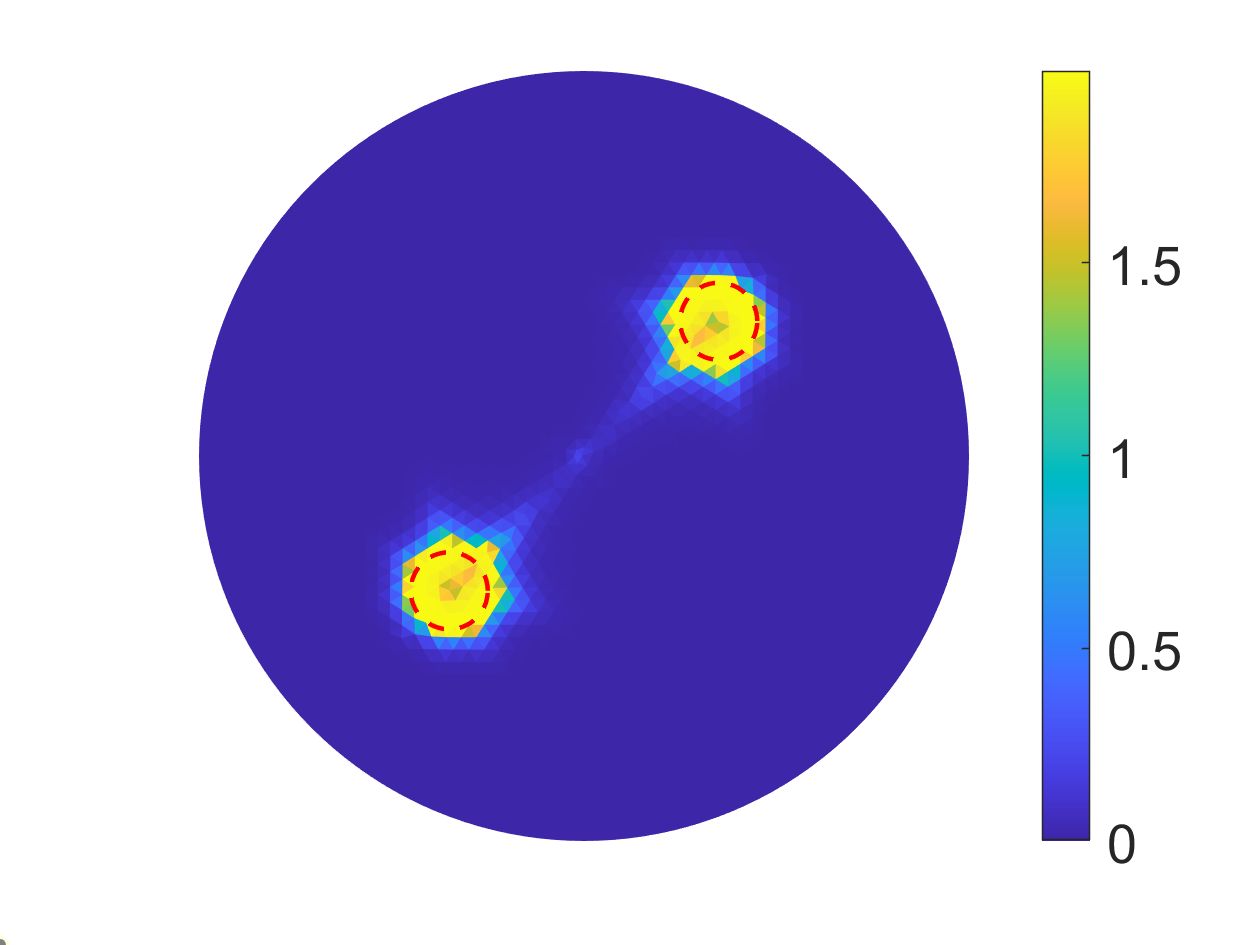}}
\subfigure[$\delta=1\%$]{\includegraphics[width=0.45\textwidth]
                   {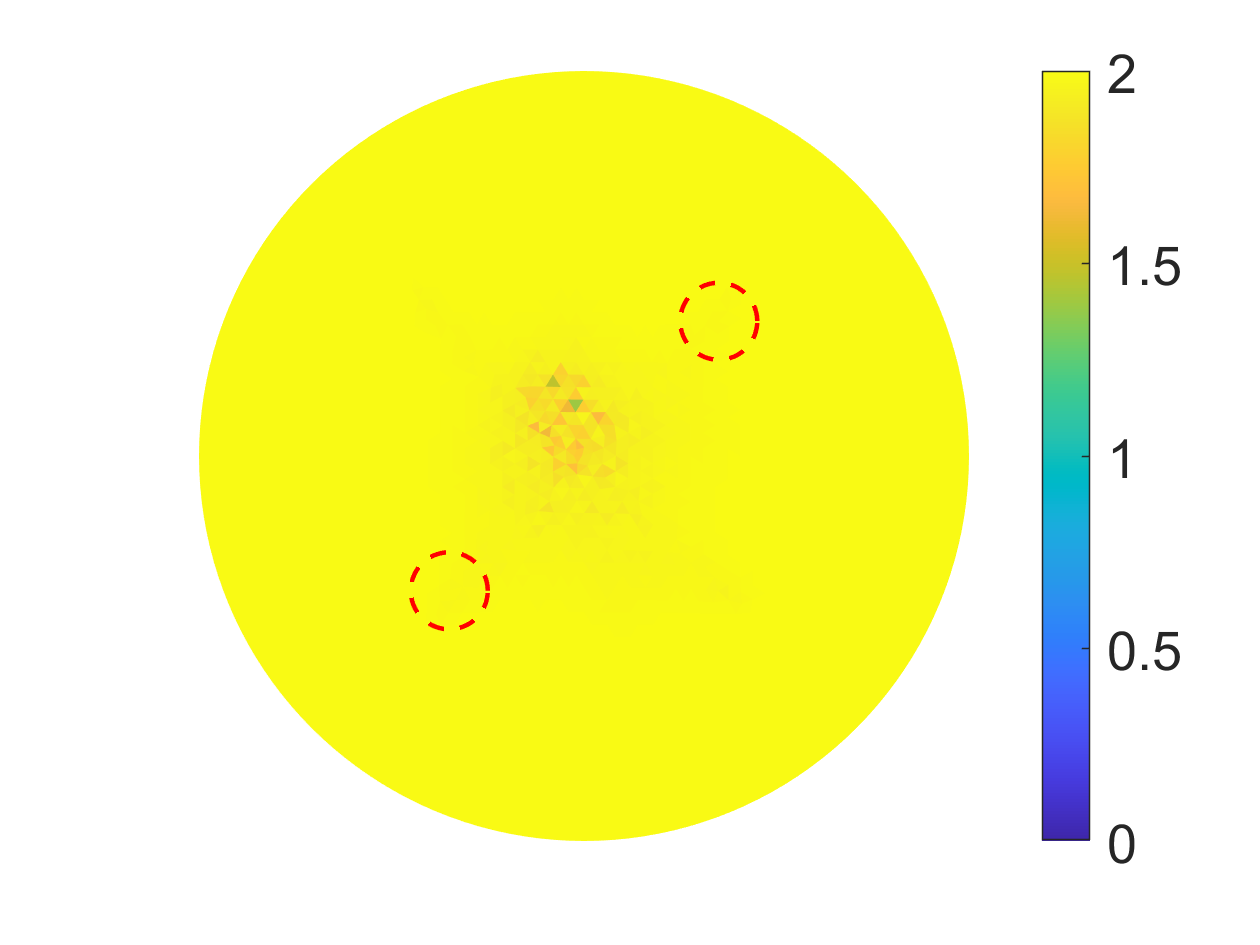}}\\
\subfigure[$\delta=10^{-9}\%$]{\includegraphics[width=0.45\textwidth]
                   {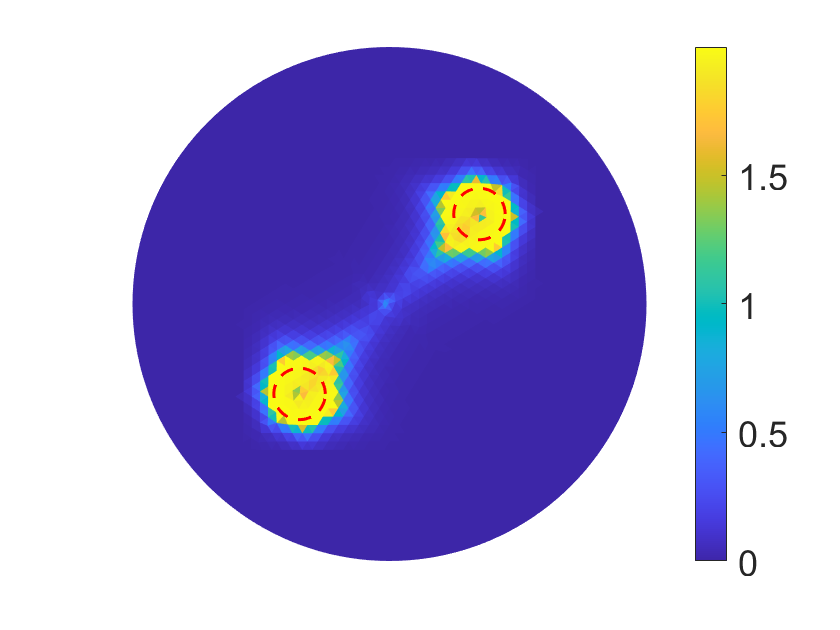}}
\subfigure[$\delta=1\%$]{\includegraphics[width=0.45\textwidth]
                   {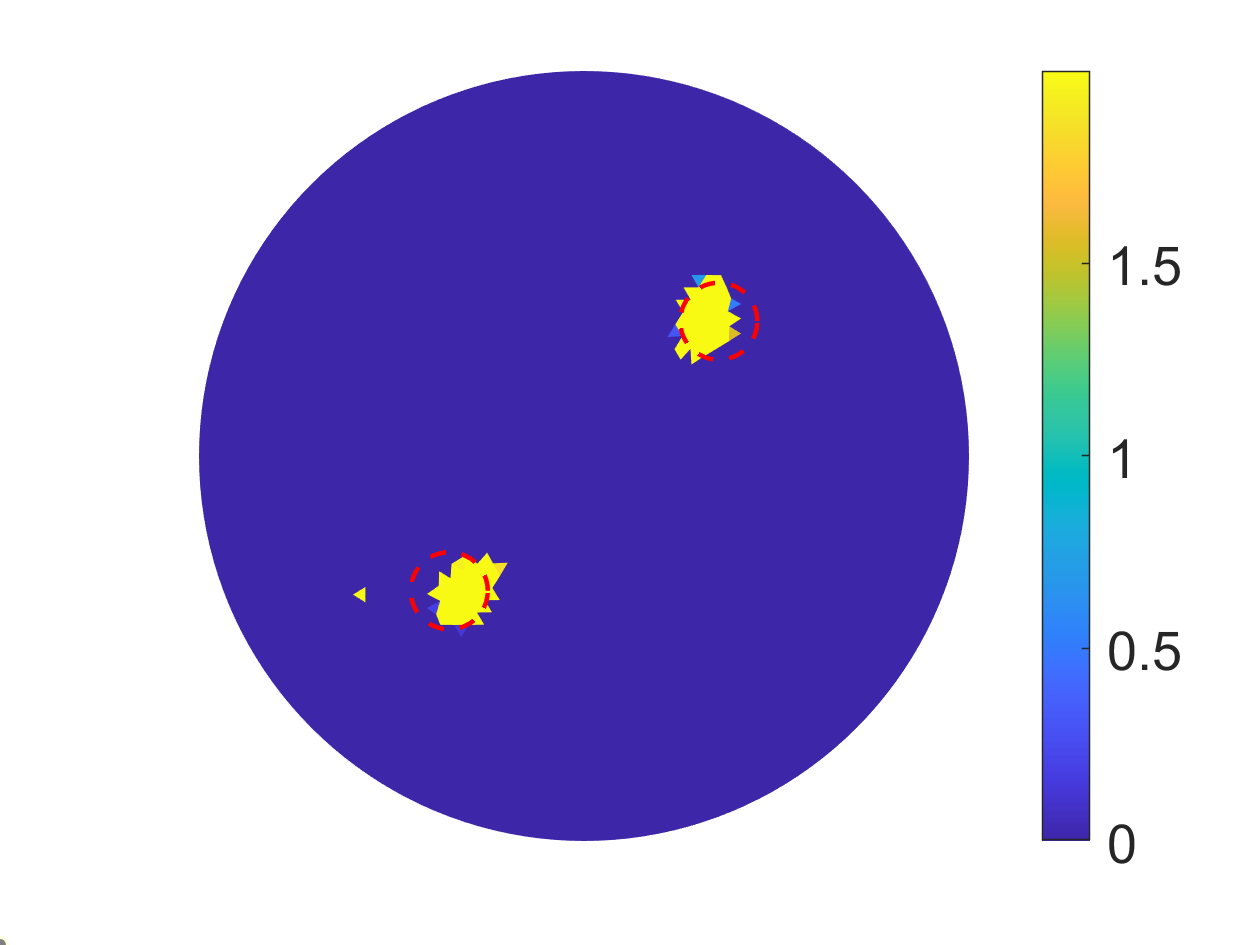}}
\caption{\label{fig:two-scatterers} (a) A schematic illustration of the geometric setting;  (b), (c)  contour plots of the support of minimizer $\hat h^{\delta}$ under the sum of all positive  eigenvalues  without penalty term;
(d), (e)  contour plots of the support of minimizer $\hat h^{\delta}$ under the sum of all positive eigenvalues with  the penalty term $\delta \|\bm R^{\delta}(h)\|_F$. }
\end{figure}

\backmatter

\bmhead{Acknowledgements}

S. Eberle-Blick thanks the German Research Foundation (DFG) for funding  the project
"Inclusion Reconstruction with Monotonicity-based Methods for the  Elasto-oscillatory
 Wave Equation" (reference number 499303971) at the Goethe-University Frankfurt, where
the major part of this article has been conducted during this project.
The work of X. Wang was supported by Alexander von Humboldt Research Fellowship for Postdocs (Ref. 3.5-1234961-CHN-HFST-P), NSFC grant 12471397 and Heilongjiang Provincial Natural Science Foundation grant YQ2024A003.


\bibliography{sn-bibliography}

\end{document}